\documentclass[12pt, reqno]{amsart} 
\usepackage{amssymb,amscd,amsfonts,amsbsy}
\usepackage{latexsym}
\usepackage{exscale}
\usepackage{amsmath,amsthm,amsfonts}
\usepackage{mathrsfs}
\usepackage{xcolor} 
\usepackage[colorlinks=true,linkcolor=blue,citecolor=blue,urlcolor=blue]{hyperref} 
\usepackage{esint} 
\usepackage{amssymb} 
\usepackage{stmaryrd}
\usepackage{cite} 
\usepackage{tikz}
\usepackage{enumerate}
\calclayout

\allowdisplaybreaks[3] 

\newtheorem{theorem}{Theorem}[section]

\newtheorem{lemma}[theorem]{Lemma}

\newtheorem{definition}[theorem]{Definition}
\newtheorem{remark}[theorem]{Remark}

\numberwithin{equation}{section}

\makeatletter
\@addtoreset{equation}{section}
\makeatother

\def\rn{{\mathbb R^n}}

\def\P{{\mathscr P}}

\def\pp{{p(\cdot)}}

\def\Xint#1{\mathchoice
   {\XXint\displaystyle\textstyle{#1}}%
   {\XXint\textstyle\scriptstyle{#1}}%
   {\XXint\scriptstyle\scriptscriptstyle{#1}}%
   {\XXint\scriptscriptstyle\scriptscriptstyle{#1}}%
   \!\int}
\def\XXint#1#2#3{{\setbox0=\hbox{$#1{#2#3}{\int}$}
     \vcenter{\hbox{$#2#3$}}\kern-.5\wd0}}

\def\avgint{\Xint-}

\def\Z{\mathbb{Z}}

\begin{document}
\title[Fractional maximal operators on...]
{\bf Fractional maximal operators on weighted variable Lebesgue spaces over the spaces of homogeneous type}

\author[X. Cen]{Xi Cen}
\address{Xi Cen\\
	School of Mathematics and Physics\\
	Southwest University of Science and Technology\\
	Mianyang 621010 \\
	People's Republic of China}\email{xicenmath@gmail.com}

\date{April 23, 2024.}




\begin{abstract} 
Let $(X,d,\mu)$ is a space of homogeneous type, we establish a new class of fractional-type variable weights $A_{p(\cdot), q(\cdot)}(X)$. Then, we get the new weighted strong-type and weak-type characterizations for fractional maximal operators $M_\eta$ on weighted variable Lebesgue spaces over $(X,d,\mu)$. This study generalizes the results by Cruz-Uribe--Fiorenza--Neugebauer \cite{Cruz2012} (2012), 
Bernardis--Dalmasso--Pradolini\cite{Ber2014} (2014), Cruz-Uribe--Shukla\cite{Cruz2018} (2018), and Cruz-Uribe--Cummings\cite{Cruz2022} (2022).
\end{abstract}

\maketitle
\tableofcontents
\section{\bf Introduction}
In this paper, we focus on the boundedness of fractional maximal operators $M_\eta$ on weighted Lebesgue spaces with variable exponents over the spaces of homogeneous type $L^{p(\cdot)}(X,\omega)$. This work is based on the theory of boundedness on weighted Lebesgue Spaces $L^{p(\cdot)}(\omega)$ and some recent work by Cruz-Uribe et al. (see Theorems {\bf  A-G} below). The theory of maximal operators was first studied by Muckenhoupt et al. \cite{Muc1972,Muc1974}, and a series of far-reaching results were obtained. Since then, the weighted theory of maximal operators can be regarded as the generalization of the work of Muckenhoupt et al.

$(\rn,|\cdot|,dx)$ is a special case of the spaces of the homogeneous type, of which we give some definitions and properties as follows.

\begin{definition}
For a positive function $d: X \times X \rightarrow [0, \infty)$, $X$ is a set, the quasi-metric space $(X, d)$ satisfies the following conditions:
\begin{enumerate}
    \item  When $x=y$, $d(x, y)=0$.
    \item $d(x, y)=d(y, x)$ for all $x, y \in X$.
    \item  For all $x, y, z \in X$, there is a constant $A_0 \geq 1$ such that $d(x, y) \leq A_0(d(x, z)+d(z, y))$.
\end{enumerate}
\end{definition}

\begin{definition}
 Let $\mu$ be a measure of a space $X$. For a quasi-metric ball $B(x,  r)$ and any $r>0$, if $\mu$ satisfies doubling condition, then there exists a doubling constant $C_\mu \geq 1$, such that  
$$
0<\mu(B(x, 2 r)) \leq C_\mu \mu(B(x, r))<\infty.
$$
\end{definition}

\begin{definition}
For a non-empty set $X$ with a qusi-metric $d$, a triple $(X, d, \mu)$ is said to be a space of homogeneous type if $\mu$ is a regular measure which satisfies doubling condition on the $\sigma$-algebra, generated by open sets and quasi-metric balls.
\end{definition}

Considering a measurable function \(p: E \rightarrow [1, \infty)\) on a subset \(E \subseteq X\), we define \(p_-(E) = \text{ess inf}_{x \in E} p(x)\) and \(p_+(E) = \text{ess sup}_{x \in E} p(x)\), with \(p_-\) and \(p_+\) specifically denoting these quantities over the entire space \(X\). Furthermore, we introduce some sets of measurable functions based on these definitions.
\begin{align*}
	&{\P}\left( E \right) = \{ p( \cdot ) :{\rm{E}} \to \left[ {1,\infty } \right) \text{ is measurable: } 1 < {p_ - }(E) \le {p_{\rm{ + }}}(E) < \infty \};\\
	&{{\P}_1}\left( {E} \right) = \{ p( \cdot ) :{\rm{E}} \to \left[ {1,\infty }\right) \text{ is measurable: } 1 \le {p_ - }(E) \le {p_{\rm{ + }}}(E) < \infty \};\\
	&{{\P}_0}\left( {E} \right) = \{ p( \cdot ) :{\rm{E}} \to \left[ {0,\infty }\right) \text{ is measurable: } 0 < {p_ - }(E) \le {p_{\rm{ + }}}(E) < \infty \}.
\end{align*}
Obviously, ${\P}\left( E \right) \subseteq {{\P}_1}\left( {E} \right) \subseteq {{\P}_0}\left( {E} \right)$. When $E=X$,  we write $\mathscr{P}(X)$ by ${\P}$ for convenience.

\begin{definition}
Let $1\leq p_-\leq p_+\leq \infty$, the variable exponent Lebesgue spaces with Luxemburg norm is defined as
\begin{equation*}
L_{}^{p( \cdot )}(X) = \{ f:{\left\| f \right\|_{L_{}^{p( \cdot )}(X)}}: = \inf \{ \lambda  > 0:{\rho _{p( \cdot )}}(\frac{f}{\lambda }) \le 1\}  < \infty \},
\end{equation*}
where ${\rho _{p( \cdot )}}(f) = \int_X {{{\left| {f(x)} \right|}^{p(x)}}dx}+\|f\|_{L^{\infty}(X_\infty)}.$ We always abbreviate $\|\cdot\|_{L^{p(\cdot)}(X)}$ to $\|\cdot\|_{p(\cdot)}$. For every ball $B \subseteq X$, if $\rho_\pp(f\chi_B) < \infty$, then $f$ is said to be locally $\pp$-integrable.
\end{definition}
 
In fact, the above spaces are Banach spaces (precisely, ball Banach function spaces), to which readers can refer \cite{red}.

Let $\omega$ be a weight function on $X$. The variable exponent weighted Lebesgue spaces are defined by
\begin{equation*}
	L_{}^{p( \cdot )}(X,\omega) = \{ f:{\left\| f \right\|_{L_{}^{p( \cdot )}(X,\omega)}} := {\left\| {\omega f} \right\|_{L_{}^{p( \cdot )}(X)}} < \infty \}.
\end{equation*}

\begin{definition}
For any $x, y \in X$ and $d(x,y)<\frac{1}{2}$, we say $p(\cdot) \in LH_{0}$, if  	\begin{equation}\label{LH0}
		|p(x)-p(y)| \lesssim \frac{1}{\log (e+1 /d(x,y))}.
	\end{equation}
We say $p(\cdot) \in LH_{\infty}$ (respect to a point $x_0 \in X$), if there exists $p_{\infty} \in X$, for any $x \in X$,
	\begin{equation}\label{LHinfty}
		\left|p(x)-p_{\infty}\right| \lesssim \frac{1}{\log (e+d(x,x_0))} .
	\end{equation}
We denote the globally log-H\"{o}lder continuous functions by $LH=LH_{0} \cap LH_{\infty}$.
\end{definition}
According to the above definition, it seems to relate to the choice of the point $x_0$. However, through \cite{Ad2015}, we can know that such a choice is immaterial.
\begin{lemma}\label{LHxy}
     For any $y_0\in X$, if $p(\cdot) \in L H_{\infty}$ with respect to $x_0\in X$, then $p(\cdot) \in L H_{\infty}$ with respect to $y_0$.
\end{lemma}

If $x_0$ is not chosen definitely, we always suppose that $X$ has an arbitrary given point $x_0$.

The fractional maximal operator $M_\eta$ on the spaces of homogeneous type is defined as
$$
M_\eta f(x)=\sup _{B \subseteq X}\mu(B)^{\eta-1} \int_B|f(y)| d \mu \cdot {\chi _B}(x).
$$
When $X=\rn$, we take $\eta=\frac{\alpha}{n}$ and write $M_\eta$ by $M_\alpha$.
We now give the definition of fractional-type weights \(A_{p(\cdot),q(\cdot)}(X)\) and present some foundational results.
\begin{definition}\label{def.Apq}
Let $p(\cdot), q(\cdot) \in \mathscr{P}_1$ and $\frac{1}{p(\cdot)}-\frac{1}{q(\cdot)}=\eta \in[0,1)$. We say a weight $\omega \in A_{p(\cdot), q(\cdot)}(X)$, if 
    $$
    [\omega]_{A_{p(\cdot), q(\cdot)}(X)}:=\sup _{B \subseteq X} \mu(B)^{\eta-1}\left\|\omega \chi_B\right\|_{{q(\cdot)}}\left\|\omega^{-1} \chi_B\right\|_{{p^{\prime}(\cdot)}}<\infty.
    $$
\end{definition}

\begin{figure}[!h]
	\begin{center}
		\begin{tikzpicture}
\node (1) at(0,0) {$A_p(\rn)$};
\node (2) at(3,3) {$A_{p}(X)$};
\node (3) at(3,-3) {$A_{p, q}(\rn)$};
\node (4) at(6,3) {$A_{p(\cdot)}(X)$};
\node (5) at(6,-3) {$A_{p,q}(X)$};
\node (6) at(6,0) {$A_{p(\cdot), q(\cdot)}(\rn)$};
\node (7) at(9,0) {$A_{p(\cdot), q(\cdot)}(X)$};
\node (8) at(3,0) {$A_{p(\cdot)}(\rn)$};
			\draw[->] (1)--(2);
			\draw[->] (1)--(3);
   		\draw[->] (1)--(8);
			\draw[->] (2)--(4);
			\draw[->] (2)--(5);
			\draw[->] (3)--(5);
            \draw[->] (3)--(6);
			\draw[->] (8)--(4);
            \draw[->] (8)--(6);
            
			\draw[->] (4)--(7);
			\draw[->] (5)--(7);
            \draw[->] (6)--(7);
		\end{tikzpicture}
		\end{center}
\caption{The relationships between weights.}\label{figure1}
	\end{figure}
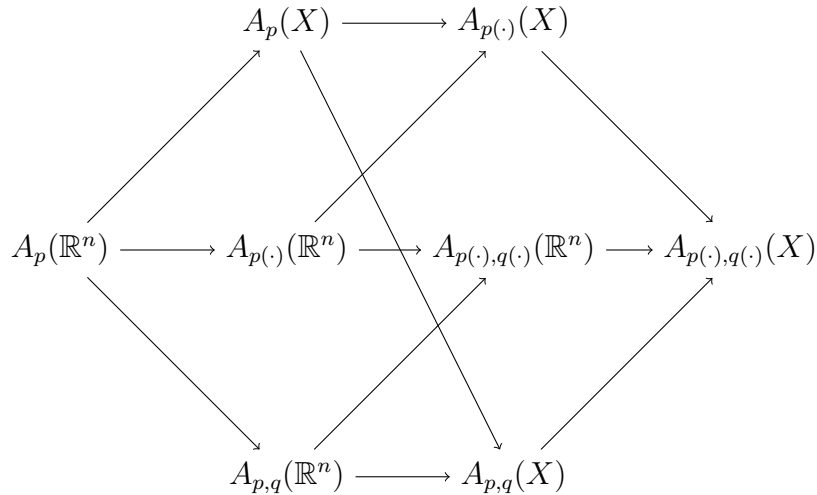

\begin{remark}
The above discussion introduce a broader category of weights, implying that the $A_{p(\cdot),q(\cdot)}$ set can be deduced as many particular instances under specific conditions $(Fig. \ref{figure1})$.
    \begin{enumerate}

\item If $\eta=0$, then $A_{p(\cdot), q(\cdot)}(X)=A_{p(\cdot)}(X)$ introduced in \cite{Cruz2022}.

\item If $p(\cdot)\equiv p$, then $A_{p(\cdot), q(\cdot)}(X)=A_{p,q}(X)$. 

\item If $p(\cdot)\equiv p$ and $\eta=0$, then $A_{p(\cdot), q(\cdot)}(X)=A_{p}(X)$. 

\item If $X=\rn$, then $A_{p(\cdot), q(\cdot)}(X)=A_{p(\cdot),q(\cdot)}(\rn)$ introduced in \cite{Ber2014}.

\item If $X=\rn$ and $\eta=0$, then $A_{p(\cdot), q(\cdot)}(X)=A_{p(\cdot)}(\rn)$ introduced in \cite{Cruz2012}.

\item If $X=\rn$ and $p(\cdot)\equiv p$, then $A_{p(\cdot), q(\cdot)}(X)=A_{p,q}(\rn)$ introduced in \cite{Muc1974}.

\item If $X=\rn$, $p(\cdot)\equiv p$, and $\eta=0$, then $A_{p(\cdot), q(\cdot)}(X)=A_{p}(\rn)$ introduced in \cite{Muc1972}.
    \end{enumerate}
\end{remark}

In this paper, we always abbreviate $A_{p(\cdot), q(\cdot)}(X)$ to $A_{p(\cdot), q(\cdot)}$, $A_{p(\cdot)}(X)$ to $A_{p(\cdot)}$, and $A_{p}(X)$ to $A_{p}$. It is easy to observe that
$$
\left[\omega^{-1}\right]_{A_{q^{\prime}(\cdot), p^{\prime}(\cdot)}}=[\omega]_{A_{p(\cdot), q(\cdot)}}.
$$
By H\"older's inequality, we have
\begin{align}
{\left[ \omega  \right]_{{A_{q( \cdot )}}}} \le& {\left[ \omega  \right]_{{A_{p( \cdot ),q( \cdot )}}}},\label{cha.1}\\ 
{\left[ {{\omega ^{ - 1}}} \right]_{{A_{p'( \cdot )}}}} \le& {\left[ \omega  \right]_{{A_{p( \cdot ),q( \cdot )}}}}.\label{cha.2}
\end{align}


We introduce some background and motivation regarding the main results of this paper.

Since 1972 and 1974, Muckenhoupt et al. \cite{Muc1972,Muc1974} studied the characterization of $A_{p}(\rn)$ and $A_{p,q}(\rn)$ by maximal operators $M$ and fractional maximal operators $M_{\alpha}$ respectively. Many people began to pay attention to the relationship between the characterization of weights and maximal operators.

In 2012, Cruz-Uribe, Fiorenza, and Neugebauer \cite{Cruz2012} firstly studied the characterization of $A_{p(\cdot)}(\rn)$ by maximal operators $M$.

\, \hspace{-20pt}{\bf Theorem A}(\cite{Cruz2012}). {\it\
	Let $p(\cdot) \in \P \cap L H$ and $\omega$ is a weight. Then $M$ is bounded on $L^{p(\cdot)}(\omega)$ if and only if $\omega \in A_{p(\cdot)}(\rn)$.
}

\, \hspace{-20pt}{\bf Theorem B}(\cite{Cruz2012}). {\it\
	Let $p(\cdot) \in \mathscr{P}_1 \cap L H$ and $\omega$ is a weight. Then
	$M$ is bounded from $L^{p(\cdot)}(\omega)$ to $WL^{p(\cdot)}(\omega)$
	if and only if $\omega \in A_{p(\cdot)}(\rn)$.
}

In 2014, Bernardis, Dalmasso, and Pradolini \cite{Ber2014} proved the characterizations for $A_{p(\cdot),q(\cdot)}(\rn)$ by fractional maximal operators $M_{\alpha}$ as follows.

\, \hspace{-20pt}{\bf Theorem C}(\cite{Ber2014}).{\it\
Let $p(\cdot), q(\cdot) \in \mathscr{P} \cap L H$, $ \frac{1}{p(\cdot)}-\frac{1}{q(\cdot)}=\frac{\alpha }{n} \in[0,1)$, and $\omega$ is a weight. Then $M_{\alpha}$ is bounded from $L^{p(\cdot)}(\omega)$ to $L^{q(\cdot)}(\omega)$ if and only if $\omega \in A_{p(\cdot), q(\cdot)}(\rn)$.
}

In 2018, Cruz-Uribe and Shukla \cite{Cruz2018} obtained the following results, which solve the problem of boundedness of fractional maximal operators on variable Lebesgue spaces over the spaces of homogeneous type.

\, \hspace{-20pt}{\bf Theorem D}(\cite{Cruz2018}).{\it\
Let $p(\cdot), q(\cdot) \in \mathscr{P} \cap L H$ and $ \frac{1}{p(\cdot)}-\frac{1}{q(\cdot)}=\eta \in[0,1)$ Then $M_\eta$ is bounded from $L^{p(\cdot)}(X)$ to $L^{q(\cdot)}(X)$.

Additionally, if \(\mu(X) < +\infty\), the requirement \( p(\cdot) \in L H\) can be substituted with \(p(\cdot) \in L H_0\).
}

\, \hspace{-20pt}{\bf Theorem E}(\cite{Cruz2018}). {\it\
Let $p(\cdot), q(\cdot) \in \mathscr{P}_1 \cap L H$ and $ \frac{1}{p(\cdot)}-\frac{1}{q(\cdot)}=\eta \in[0,1)$ Then $M_\eta$ is bounded from $L^{p(\cdot)}(X)$ to $WL^{q(\cdot)}(X)$.

Moreover, if \(\mu(X) < +\infty\), the condition \( p(\cdot) \in L H\) may be substituted with \(p(\cdot) \in L H_0\).
}

In 2022, Cruz-Uribe and Cummings \cite{Cruz2022} demonstrated the following characterizations for $A_{p(\cdot)}(X)$ by maximal operators $M$.

\, \hspace{-20pt}{\bf Theorem F}(\cite{Cruz2022}).\label{VApp1} {\it\
Let $p(\cdot) \in \P \cap L H$ and $\omega$ is a weight. Then $M$ is bounded on $L^{p(\cdot)}(X, \omega)$ if and only if $\omega \in A_{p(\cdot)}(X)$.
}

\, \hspace{-20pt}{\bf Theorem G}(\cite{Cruz2022}).\label{VApp2} {\it\
Let $p(\cdot) \in \mathscr{P}_1 \cap L H$ and $\omega$ is a weight. Then
$M$ is bounded from $L^{p(\cdot)}(X, \omega)$ to $WL^{p(\cdot)}(X, \omega)$
if and only if $\omega \in A_{p(\cdot)}(X)$.
}

Inspired by the above, it is natural to consider whether $A_{p(\cdot),q(\cdot)}(X)$ can be characterized by fractional maximal operators $M_\eta$? The answer to this question is yes. To be precise, we can draw the following conclusions.

\begin{theorem}\label{mainthm_0}
Let $p(\cdot), q(\cdot) \in \mathscr{P} \cap L H$, $ \frac{1}{p(\cdot)}-\frac{1}{q(\cdot)}=\eta \in[0,1)$, and $\omega$ is a weight. Then $M_\eta$ is bounded from $L^{p(\cdot)}(X, \omega)$ to $L^{q(\cdot)}(X, \omega)$ if and only if $\omega \in A_{p(\cdot), q(\cdot)}(X)$.
\end{theorem}

\begin{theorem}\label{mainthm_1}
Let $p(\cdot), q(\cdot) \in \mathscr{P} \cap L H$, $ \frac{1}{p(\cdot)}-\frac{1}{q(\cdot)}=\eta \in[0,1)$, and $\omega$ is a weight. Then
$M_\eta$ is bounded from $L^{p(\cdot)}(X, \omega)$ to $WL^{q(\cdot)}(X, \omega)$
if and only if $\omega \in A_{p(\cdot), q(\cdot)}(X)$.
\end{theorem}

\begin{remark}
It is obvious that Theorems \ref{mainthm_0} and \ref{mainthm_1} generalizes Theorems {\bf  A-G}. 
\end{remark}
We still need introduce some notations which will be used in this paper.

For some positive constant $C$ independent of appropriate parameters, $A\lesssim B$ means that $A\leq CB$ and $ A \approx B$ means that $A\lesssim B$ and $B\lesssim A$. What's more $A \lesssim_{\alpha ,\beta } B$ means that $A\leq C_{\alpha,\beta}B$, where $C_{\alpha,\beta}$ is dependent on $\alpha,\beta$.
Given an open set $E \subseteq {\rn}$ and a measurable function $p( \cdot ): {\rm{E}} \to \left[ {1,\infty } \right)$, $p'( \cdot )$ is the conjugate exponent defined by $p'( \cdot ) = \frac{{p( \cdot )}}{{p( \cdot ) - 1}}$. 
A weight is defined as a locally integrable function $\omega: X \rightarrow [0, \infty]$ satisfying $0 < \omega(x) < \infty$ for almost every $x \in X$. For a given weight $\omega$, its associated measure is established as $d\omega(x) = \omega(x) d\mu(x)$.For a subset $E\subseteq X$, the weighted average integral of a function $f$ is represented by
$$
\avgint_E f(x) d \omega=\frac{1}{\omega(E)} \int_E f(x) \omega(x) d \mu .
$$

The structure of this paper is as follows.
In Section \ref{pre}, we give some lemmas for variable Lebesgue spaces, weights, and dyadic analysis respectively, which play a important roles for the proof of our main theorems. In Section \ref{proof}, we prove Theorems \ref{mainthm_0} and \ref{mainthm_1}.

\section{\bf Preliminaries}\label{pre}

\subsection{Variable Lebesgue spaces}\label{2.1}
~

This subsection includes some foundational lemmas of variable Lebesgue spaces over the spaces of the homogeneous type. The first lemma is called "Lower Mass Bound".

\begin{lemma}[\cite{Cruz2022}, Lemma 2.1]\label{LMB.}
	For all $0 < r < R$ and any $y \in B(x, R)$, there exists a positive constant $C=C_{X}$, such that
	$$
	\frac{\mu(B(y, r))}{\mu(B(x, R))} \geq C\left(\frac{r}{R}\right)^{\log _2 C_\mu}.
	$$
\end{lemma}
\begin{lemma}[\cite{Bra1996}, Lemma 1.9]\label{finite-bounded}
	$\mu(X) < \infty$ if and only if $X$ is bounded, which means there exist a ball $B \subseteq X$, such that $X=B$.
\end{lemma}

\begin{lemma}[\cite{Cruz2022}, Lemma 2.11]\label{2Log_3}
	Let $p(\cdot) \in L H$, then 
	$
	\mathop {\sup }\limits_{B \subseteq X} \mu {(B)^{{p_ - }(B) - {p_ + }(B)}} \lesssim 1.
	$
\end{lemma}

The following lemmas are initial parts supporting our main results. Actually, some lemmas' proofs are identical to their Euclidean case, and readers can refer to \cite{red,Cruz2022,Di2011} for more details. Hence, we will omit some of them for the brevity of this paper.

\begin{lemma}[\cite{red}, Proposition 2.21]\label{unity}
	Let $p(\cdot) \in \P_1$, then
	$$
	\int_X\left(\frac{|f(x)|}{\|f\|_{p(\cdot)}}\right)^{p(x)} d \mu=1 .
	$$
\end{lemma}
\begin{lemma}[\cite{red}, Corollary 2.23] \label{p.omega}
	Let $\Omega \subseteq X$ and $p(\cdot) \in \P_1 (\Omega)$. 
 
    If $\|f\|_{L^{p(\cdot)}(\Omega)} \leq 1$, then
	$$
	\|f\|_{p(\cdot)}^{p_{+}(\Omega)} \leq \int_\Omega|f(x)|^{p(x)} d \mu \leq\|f\|_{p(\cdot)}^{p_{-}(\Omega)} .
	$$
 
	If $\|f\|_{L^{p(\cdot)}(\Omega)} \geq 1$, then
	$$
	\|f\|_{p(\cdot)}^{p_{-}(\Omega)} \leq \int_\Omega|f(x)|^{p(x)} d \mu \leq\|f\|_{p(\cdot)}^{p_{+}(\Omega)} .
	$$
	Moverover, we have $\|f\|_{p(\cdot)} \leq C_1$ if and only if
	$\int_\Omega|f(x)|^{p(x)} d \mu \leq C_2$.
	When either $C_1=1$ or $C_2=1$, the other constant is also to be 1.
\end{lemma}

\begin{lemma}[\cite{Cruz2022}, Lemma 2.6]\label{density}
	Let $p(\cdot) \in \P_1$, then the bounded functions with bounded support are dense in $L^{p(\cdot)}(X)$. Furthermore, any nonnegative function $f$ in $L^{p(\cdot)}(X)$ can be approximated as the limit of an increasing sequence.
\end{lemma}
\begin{lemma}[\cite{red}, Theorem 2.59]\label{Levi}
	Let $p(\cdot) \in \P_1$. For a sequence of non-negative measureable functions, denoted as $\left\{f_k\right\}_{k=1}^{\infty}$ and increasing pointwise almost everywhere to a function $f \in L^{p(\cdot)}$, we can deduce that $\left\|f_k\right\|_{p(\cdot)} \rightarrow\|f\|_{p(\cdot)}$.
\end{lemma}

\begin{lemma}[\cite{Cruz2022}, Lemma 2.10]\label{2Log_2}
	For any point $y\in G$, $G$ is a subset of $X$, and two exponents $p_1(\cdot)$ and $p_2(\cdot)$, there exists a constant $C_0>0$ such that
	$$
	|p_1(y)-p_2(y)| \leq \frac{C_0}{\log \left(e+d\left(x_0, y\right)\right)}.
	$$
	Then there exists a constant $C=C_{t,C_0}$ such that 
	\begin{equation}\label{2Log_1}
	\int_G|f(y)|^{p_1(y)} u(y) d \mu \leq C \int_G|f(y)|^{p_2(y)} u(y) d \mu+\int_G \frac{1}{\left(e+d\left(x_0, y\right)\right)^{t s_{-}(G)}} u(y) d \mu
	\end{equation}
	for all functions $f$ with $|f(y)| \leq 1$ and every $t\geq 1$.
\end{lemma}

\subsection{Properties of weights}
~

This subsection is aimed to exploring the properties of the $A_{p(\cdot), q(\cdot)}$ condition within spaces of homogeneous type. The following lemma reflects the properties of $A_{\infty}$, defined by $\bigcup_{p \geq 1} A_p$, whose proof are similar to that of \cite[Theorem 7.3.3]{249}.
\begin{lemma}\label{Ainfty}
If $\omega$ is a weight function, then the following conditions are equivalent:
    \begin{enumerate}
        \item $\omega \in A_{\infty}$.
        \item There exist constants $\epsilon>0$ and $C_2>1$ such that 
        $$
        \frac{\mu(E)}{\mu(B)} \leq C_2\left(\frac{\omega(E)}{\omega(B)}\right)^\epsilon,
        $$
        for any ball $B$ and 
        its measurable subset $E$.
        \item The measure $d \omega(x)=\omega(x) d \mu(x)$ satisfies doubling condition and there exist constants $\delta>0$ and $C_1>1$ such that 
        $$
        \frac{\omega(E)}{\omega(B)} \leq C_1\left(\frac{\mu(E)}{\mu(B)}\right)^\delta,
        $$
        for any ball $B$ and its  measurable subset $E$.
    \end{enumerate}
    \end{lemma}
The following Hölder's inequality is very useful.
\begin{lemma}[\cite{red}, Theorem 2.26]\label{Holder}
Let $p(\cdot) \in \P_1$, then
$$
\int_X|f(x) g(x)| d \mu \leq 4\|f\|_{p(\cdot)}\|g\|_{p^{\prime}(\cdot)}.
$$
\end{lemma}
To apply the properties introduced in the above, this study employs the $A_{p(\cdot), q(\cdot)}$ condition for the construction of a weight $W$, see Lemma \ref{Apq_Ainfty}, within the $A_{\infty}$ class. 
And the following lemmas are necessary for this purpose.
    \begin{lemma}\label{normbound}
    Let $p(\cdot), q(\cdot) \in \mathscr{P}_1$ and $\frac{1}{p(\cdot)}-\frac{1}{q(\cdot)}=\eta \in[0,1)$. For any ball $B$ and its measurable subset $E$, if $\omega \in A_{p(\cdot), q(\cdot)}$, then 
    $$
\left(\frac{\mu(E)}{\mu(B)}\right)^{1-\eta} \le16[\omega]_{A_{p(\cdot), q(\cdot)}} \frac{\left\|\omega \chi_E\right\|_{q(\cdot)}}{\left\|\omega \chi_B\right\|_{q(\cdot)}}.
    $$
    \end{lemma}
    \begin{proof}
    By Hölder's inequality and the $A_{p(\cdot), q(\cdot)}$ condition (Definition \ref{def.Apq}),
    \begin{align*}
    \mu(E) & =\int_X \omega(x) \chi_E \omega(x)^{-1} \chi_B d \mu \\
    & \leq 4\left\|\omega \chi_E\right\|_{q(\cdot)}\left\|\omega^{-1} \chi_B\right\|_{q^{\prime}(\cdot)} \\
    & \leq 16 \mu(E)^\eta\left\|\omega \chi_E\right\|_{q(\cdot)}\left\|\omega^{-1} \chi_B\right\|_{p^{\prime}(\cdot)}. \\
    \end{align*}
    Thus,
    $$
\left(\frac{\mu(E)}{\mu(B)}\right)^{1-\eta} \le16[\omega]_{A_{p(\cdot), q(\cdot)}} \frac{\left\|\omega \chi_E\right\|_{q(\cdot)}}{\left\|\omega \chi_B\right\|_{q(\cdot)}}
    $$
    \end{proof}
The next lemma plays a important role in our proof, which is dedicated to the proof of \eqref{Suff.ineq_7}.    
    \begin{lemma}[\cite{Cruz2022}, Lemma 3.3]\label{fracexp}
Let $p(\cdot) \in \mathscr{P}_1\cap LH$ and $\omega \in A_{p(\cdot)}$. Then
    $$
\mathop {\sup }\limits_{B \subseteq X} \left\| {\omega {\chi _B}} \right\|_{p( \cdot )}^{{p_ - }(B) - {p_ + }(B)} \lesssim 1.
    $$
    \end{lemma}
   

\begin{lemma}\label{Apq_Ainfty}
Let $p(\cdot), q(\cdot) \in \mathscr{P}_1\cap LH$, $\frac{1}{p(\cdot)}-\frac{1}{q(\cdot)}=\eta \in [0,1)$, and $\omega \in A_{p(\cdot), q(\cdot)}$. Then $W(\cdot):=$ $\omega(\cdot)^{q(\cdot)} \in A_{\infty}$.
    \end{lemma}
    \begin{proof}
    Fix a ball $B$ and a measurable set $E \subseteq B$. According to Lemma \ref{Ainfty}, in order to proof this lemma, it is sufficient to prove
    \begin{equation}\label{Apq.Ainfty_1}
        \left(\frac{\mu(E)}{\mu(B)}\right)^{1-\eta} \lesssim \left(\frac{W(E)}{W(B)}\right)^{1 / q_{+}}.
    \end{equation}
    We will prove this in three cases. 
    \begin{enumerate}[(i)]
       
        \item When $\left\|\omega \chi_B\right\|_{q(\cdot)} \leq 1$. By Lemma \ref{normbound},
    $$
    \left(\frac{\mu(E)}{\mu(B)}\right)^{1-\eta} \lesssim \frac{\left\|\omega \chi_E\right\|_{q(\cdot)}}{\left\|\omega \chi_B\right\|_{q(\cdot)}} =  \frac{\left\|\omega \chi_E\right\|_{q(\cdot)}}{\left\|\omega \chi_B\right\|_{q(\cdot)}^{q_{-}(B) / q_{+}(B)}\left\|\omega \chi_B\right\|_{p(\cdot)}^{1-q_{-}(B) / q_{+}(B)}} .
    $$
    From Lemma \ref{p.omega}, we have that $\left\|\omega \chi_E\right\|_{q(\cdot)} \leq W(E)^{1 / q_{+}(B)}$ and $\left\|\omega \chi_B\right\|_{q(\cdot)}^{q_-(B)} \geq W(B)$. It follows from Lemma \ref{fracexp} and \eqref{cha.1} that
    $$
    \left(\frac{\mu(E)}{\mu(B)}\right)^{1-\eta} \lesssim \left(\frac{W(E)}{W(B)}\right)^{1 / q_{+}(B)}\left\|\omega \chi_B\right\|_{q(\cdot)}^{q_{-}(B) / q_{+}(B)-1} \lesssim \left(\frac{W(E)}{W(B)}\right)^{1 / q_{+}}
    .$$
    \item When $\left\|\omega \chi_E\right\|_{q(\cdot)} \leq 1 \leq\left\|\omega \chi_B\right\|_{q(\cdot)}$, by Lemmas \ref{normbound} and \ref{p.omega} again, we have
    \begin{align*}
    \left(\frac{\mu(E)}{\mu(B)}\right)^{1-\eta} \lesssim \frac{\left\|\omega \chi_E\right\|_{q(\cdot)}}{\left\|\omega \chi_B\right\|_{q(\cdot)}} 
    \lesssim \frac{{W{{(E)}^{1/{q_ + }}}}}{{W{{(B)}^{1/{q_ + }(B)}}}}
    \le {\left( {\frac{{W(E)}}{{W(B)}}} \right)^{\frac{1}{{{q_ + }}}}}.
    \end{align*}

\item When $1<\|\omega \chi_E\|_{q(\cdot)} \leq \|\omega \chi_B\|_{q(\cdot)}$, define $\lambda=\|\omega \chi_B\|_{q(\cdot)} \geq \|\omega \chi_E\|_{q(\cdot)}$ and substitute the measure $d\mu$ with $W(x)d\mu$. Through Lemma \ref{2Log_2}, there is a constant $C_t$ satisfies
\begin{equation}\label{Apq.Ainfty_2}
    \int_B \frac{W(x)}{\lambda^{q_{\infty}}} d \mu \leq C_t \int_B \frac{W(x)}{\lambda^{q(x)}} d \mu + \int_B \frac{W(x)}{(e+d(x_0, x))^{t q_{\infty}}} d \mu.
\end{equation}
By Lemma \ref{unity}, we can know that the first term on the right side is less than 1. Therefore we now need to prove the second term also satisfies this bound, when we take large enough $t$, independent of $B$. For a finite $W(X)$,
\begin{equation*}
\int_X \frac{W(x)}{(e+d(x_0, x))^{t q_{\infty}}} d \mu \leq Ce^{-t q_{\infty}} W(X).
\end{equation*}
If $W(X)=\infty$, we define $B_k=B(x_0, 2^k)$ and it follows from Lemmas \ref{p.omega} and \ref{Levi} that $\mathop {\lim }\limits_{k \to \infty } {\left\| {\omega {\chi _{{B_k}}}} \right\|_{p( \cdot )}} = \infty$. Lemma \ref{p.omega} provides
\begin{align*}
\int_X \frac{W(x)}{(e+d(x_0, x))^{t q_{\infty}}} d \mu &\lesssim_t e^{-t q_{\infty}} W(B_0) + \sum_{k=1}^{\infty} \int_{B_k \setminus B_{k-1}} \frac{W(x)}{(e+d(x_0, x))^{t q_{\infty}}} d \mu \\
&\le e^{-t q_{\infty}} W(B_0) + \sum_{k=1}^{\infty} 2^{-ktq_{\infty}} W(B_k) \\
&\le e^{-t q_{\infty}} W(B_0) + \sum_{k=1}^{\infty} 2^{-ktq_{\infty}} \max \{\|\omega \chi_{B_k}\|_{q(\cdot)}^{q_{+}}, \|\omega \chi_{B_k}\|_{q(\cdot)}^{q_{-}}\}\\
&\lesssim e^{-t q_{\infty}} W(B_0) + \sum_{k=1}^{\infty} 2^{-ktq_{\infty}} \|\omega \chi_{B_k}\|_{q(\cdot)}^{q_{+}},
\end{align*}
where the last inequality is derived from this fact that since $\mathop {\lim }\limits_{k \to \infty } {\left\| {\omega {\chi _{{B_k}}}} \right\|_{p( \cdot )}} = \infty$, then there exists $N>0$, for any $k>N$, we have ${\left\| {\omega {\chi _{{B_k}}}} \right\|_{p( \cdot )}} > 1$.
By Lemma \ref{normbound},
\begin{equation*}
\|\omega \chi_{B_k}\|_{q(\cdot)} \leq C\left(\frac{\mu(B_k)}{\mu(B_0)}\right)^{1-\eta}\|\omega \chi_{B_0}\|_{q(\cdot)} \leq C 2^{k {(1-\eta)}\log_2 C_\mu}.
\end{equation*}
Hence, we have
\begin{equation}\label{Apq.Ainfty_3}
    \int_X \frac{W(x)}{(e+d(x_0, x))^{t q_{\infty}}} d \mu \lesssim e^{-t q_{\infty}} W(B_0) +  \sum_{k=1}^{\infty} 2^{kq_{+} {(1-\eta)}\log_2 C_\mu - ktq_{\infty}}.
\end{equation}
When $t > \frac{{{q_ + }(1 - \eta ){{\log }_2}{C_\mu }}}{{{q_\infty }}}$, the sum is convergent. The right-hand side of \eqref{Apq.Ainfty_2} becomes bounded, which means that
\begin{equation}\label{Apq.Ainfty_4}
    W(B)^{1/q_{\infty}} \lesssim \|\omega \chi_B\|_{q(\cdot)}.
\end{equation}
Replacing $B$ by $E$ and $q(\cdot)$ by $q_{\infty}$, we get
\begin{equation*}
1 \leq \int_E \frac{W(x)}{\lambda^{q(x)}} d \mu \leq C_t \int_E \frac{W(x)}{\lambda^{q_{\infty}}} d \mu + \int_E \frac{W(x)}{(e+d(x_0, x))^{t q_{\infty}}} d \mu.
\end{equation*}
It follows from the above that   \begin{equation}\label{Apq.Ainfty_5} \lambda^{q_{\infty}}=\left\|\omega \chi_E\right\|_{q(\cdot)}^{q_{\infty}} \lesssim W(E) .
    \end{equation}
Then by Lemma \ref{normbound},
    $$
    \left(\frac{\mu(E)}{\mu(B)}\right)^{1-\eta} \lesssim \frac{\left\|\omega \chi_E\right\|_{q(\cdot)}}{\left\|\omega \chi_B \right\|_{q(\cdot)}} \lesssim \left(\frac{W(E)}{W(B)}\right)^{1 / q_{\infty}} \le \left(\frac{W(E)}{W(B)}\right)^{1 / q_{+}} .
    $$
    \end{enumerate}
    \end{proof}

It follows instantly from the proof of Lemma \ref{Apq_Ainfty} that 
   \begin{lemma}\label{chara.Apq}
Let $p(\cdot),q(\cdot) \in \P_1 \cap L H$ and $ \frac{1}{p(\cdot)}-\frac{1}{q(\cdot)}=\eta \in[0,1)$. If $\omega \in A_{p(\cdot),q(\cdot)}$ satisfying $\left\|\omega \chi_B\right\|_{q(\cdot)} \geq 1$ for some ball $B$, then $\left\|\omega \chi_B\right\|_{q(\cdot)} \approx W(B)^{1 / q_{\infty}}$.
   \end{lemma}

\begin{lemma}\label{1}
	Let $p(\cdot),q(\cdot) \in \P_1 \cap L H$ and $ \frac{1}{p(\cdot)}-\frac{1}{q(\cdot)}=\eta \in[0,1)$, then $1 \in A_{p(\cdot),q(\cdot)}$.
\end{lemma}
\begin{proof}
If $\mu (B) \le 1$, it follows from Lemma \ref{p.omega} that 
$\left\| {{\chi _B}} \right\|_{q( \cdot )}^{{q_ + }(B)} \le \mu (B)$ and $\left\| {{\chi _B}} \right\|_{p'( \cdot )}^{{{\left( {p'} \right)}_ + }(B)} \le \mu (B)$.
By Lemma \ref{2Log_3}, 
\[\mu {(B)^{\eta  - 1}}{\left\| {{\chi _B}} \right\|_{q( \cdot )}}{\left\| {{\chi _B}} \right\|_{p'( \cdot )}} \le \mu {(B)^{\frac{1}{{{q_ + }(B)}} + \eta  - \frac{1}{{{p_ + }(B)}}}} = \mu {(B)^{\frac{{{p_{+}}(B){\rm{ - }}{p_ - }(B)}}{{{p_ + }(B){p_{\rm{ - }}}(B)}}}} \le C.\]

If $\mu (B) > 1$, it follows from the proof of Lemma \ref{chara.Apq} that
\[\mu {(B)^{\eta  - 1}}{\left\| {{\chi _B}} \right\|_{q( \cdot )}}{\left\| {{\chi _B}} \right\|_{p'( \cdot )}} \le C.\]

\end{proof}

\begin{remark}
In the proof of the main theorems, we will always combine the above lemmas with \eqref{cha.1} and \eqref{cha.2} to apply it.
\end{remark}

\subsection{Dyadic Analysis}\label{2.2}
~

This classical dyadic cubes defined as
$$
Q = {2^k}([0,1)^n + m), \quad k\in \mathbb{Z}, m\in \Z^n.
$$
These constructs play an essential role in constructing our main theorem. The following discussion adopts the framework of dyadic cubes as formulated by Hytönen and Kairema \cite{Hy2012}, as explicated in \cite{And2015}.
\begin{lemma}[\cite{And2015}, Theorem 2.1]\label{cubes}
    There exist a family $\mathcal{D}=\bigcup_{k \in \mathbb{Z}} \mathcal{D}_k$, composed of subsets of $X$,  such that:
    \begin{enumerate}
        \item For cubes $Q_1, Q_2 \in \mathcal{D}$, either $Q_1 \cap Q_2=\varnothing$, $Q_1 \subseteq Q_2$, or $Q_2 \subseteq Q_1$.
        \item The cubes $Q \in \mathcal{D}_k$ are pairwise disjoint. And for any $k\in \mathbb{Z}$, $X=\bigcup_{Q \in \mathcal{D}_k} Q$. We call $\mathcal{D}_k$ as the $k$th generation.
        
        \item For any $Q_1\in \mathcal{D}_k$, there always exists at least one child of $Q_1$ in $\mathcal{D}_{k+1}$, such that $Q_2 \subseteq Q_1$, and there always exists exactly one parent of $Q_1$ in $\mathcal{D}_{k-1}$, such that $Q_1 \subseteq Q_3$.
        \item If $Q_2$ is a child of $Q_1$, then for a constant $0<\epsilon<1$, depended on the set $X$, $\mu\left(Q_2\right) \geq \epsilon \mu\left(Q_1\right)$. 
        \item For every $k\in \mathbb{Z}$ and $Q \in \mathcal{D}_k$, there exists constants $C_d$ and $d_0>1$, such that
        $$
        B\left(x_c(Q), d_0^k\right) \subseteq Q \subseteq B\left(x_c(Q), C_d d_0^k\right),
        $$
        where $x_c$ denotes the centre of cube $Q\in \mathcal{D}$.
    \end{enumerate}
    
    We call the family $\mathcal{D}$ as dyadic grid and the cubes $Q\in \mathcal{D}$ as dyadic cubes.
\end{lemma}
Frequently, the sets of cubes and balls are interchangeable, as demonstrated by the equivalent formulation of the $A_{p(\cdot),q(\cdot)}$ condition.
\begin{lemma}
 \label{Apqcube}
Let $p(\cdot), q(\cdot) \in \mathscr{P}_0 \cap LH$, $\frac{1}{p(\cdot)}-\frac{1}{q(\cdot)}=\eta \in[0,1)$, and $\mathcal{D}$ is a dyadic grid. If $\omega \in A_{p(\cdot), q(\cdot)}$, then $\omega \in {A_{p( \cdot ),q( \cdot )}^{\mathcal D}}$, that is,
$$
{\left[ \omega  \right]_{A_{p( \cdot ),q( \cdot )}^{\mathcal D}}}: = \mathop {\sup }\limits_{Q \in {\mathcal D}} \mu {(Q)^{\eta  - 1}}{\left\| {\omega {\chi _Q}} \right\|_{q( \cdot )}}{\left\| {{\omega ^{ - 1}}{\chi _Q}} \right\|_{p'( \cdot )}} < \infty.
$$
\end{lemma}
\begin{proof}
Using Theorem \ref{cubes} with fixing $Q \in \mathcal{D}_k$ and Lemma \ref{LMB.},
\begin{align*}
& \left\|\omega \chi_Q\right\|_{q(\cdot)}\left\|\omega^{-1} \chi_Q\right\|_{p^{\prime}(\cdot)} \leq\left\|\omega \chi_{B\left(x_c(Q), C_d d_0^k\right)}\right\|_{q(\cdot)}\left\|\omega^{-1} \chi_{B\left(x_c(Q), C_d r d_0^k\right)}\right\|_{p^{\prime}(\cdot)} \\
& \lesssim \mu\left(B\left(x_c(Q), C d_0^k\right)\right)^{1-\eta} \lesssim \mu\left(B\left(x_c(Q), d_0^k\right)\right)^{1-\eta} \lesssim \mu(Q)^{1-\eta} .
\end{align*}
\end{proof}
In the proof of Lemma \ref{Apqcube}, we initially expand cubes to encompass balls, subsequently applying the lower mass bound (see Lemma \ref{LMB.}) to switch back to cube dimensions. Additionally, this approach allows for the maximal operator to be efficiently reformulated in dyadic terms.
\begin{definition}
Let $\eta \in[0,1)$, $\sigma$ is a weight, and $\mathcal{D}$ is a dyadic grid. Define the weighted dyadic fractional maximal operator $M_{\eta, \sigma}^{\mathcal{D}}$ by
$$
M_{\eta, \sigma}^{\mathcal{D}} f(x)=\sup _{\substack{x \in Q \in \mathcal{D}}}{\sigma(Q)}^{\eta-1} \int_Q|f(y)| \sigma d\mu.
$$
When $\eta=0, M_{0, \sigma}^{\mathcal{D}}=M_\sigma^{\mathcal{D}}$, which is a weighted dyadic maximal operator.
When $\sigma=1, M_{\eta, \sigma}^{\mathcal{D}}=M_\eta^{\mathcal{D}}$, which is a dyadic fractional maximal operator.
\end{definition}

The following lemma can guarantee that we always transform a proof involving $M_{\eta}$ into that for $M_{\eta}^{\mathcal D_i}$.
\begin{lemma}[\cite{Kai2013}, Lemma 7.8]\label{Suff_1}
Let $\eta \in[0,1)$, there exists a finite family $\left\{\mathcal{D}_i\right\}_{i=1}^N$ of dyadic grids such that
	$$
	M_\eta f(x) \approx \sum_{i=1}^N M_\eta^{\mathcal{D}_i} f(x),
	$$
	where the implicit constants depend only $X$, $\mu$, and $\eta$.
\end{lemma}

The following lemma is in \cite{Cruz2022}, which is a key tool used in after proof.
\begin{lemma}[\cite{Cruz2022}, Lemma 4.4]\label{M.eta.bound}
Let $\mathcal{D}$ is a dyadic grid, $\sigma$ is a weight, and $1<p<\infty$. Then the dyadic maximal operator $M_{\sigma}^{\mathcal{D}}$ is bounded on $L^p(X,\sigma)$, which is also bounded from $L^{1}(X,\sigma)$ to $WL^{1}(X,\sigma)$.
\end{lemma}

We now present the fractional-type Calderón-Zygmund decomposition on the spaces of homogeneous type as follows.
\begin{lemma}\label{CZD}
Let $\eta \in[0,1)$, $\mathcal{D}$ is a dyadic grid on $X$, and $\sigma \in A_{\infty}$.
Set $\mu(X) = \infty$. If $f \in L_{\text{loc}}^1(\sigma)$ satisfying $\mathop {\lim }\limits_{j \to \infty } \sigma {\left( {Q_j} \right)^{\eta  - 1}}\int_{{Q_j}} {\left| f \right|\sigma d\mu =0}$ for any nested sequence $\left\{Q_j \in \mathcal{D}\right\}_{j=1}^{\infty}$, where $Q_{j}$ is a child of $Q_{j+1}$, then for any $\lambda > 0$, there exists a (possibly empty) collection of mutually disjoint dyadic cubes $\left\{ {Q_j^{}} \right\}$, called Calderón-Zygmund cubes for $f$ at the height $\lambda$, and a constant $C_{CZ} > 1$, which is independent of $\lambda$ and dependent of $\mathcal{D}, X, \sigma$, such that
$$
X_{\eta, \lambda}^{\mathcal{D}}:=\left\{x \in X: M_{\eta,\sigma}^{\mathcal{D}} f(x)>\lambda\right\} =\bigcup_j Q_j .
$$
Moreover, for each $j$,
\begin{align}\label{CZ_1}
\lambda<\sigma {\left( {{Q_j}} \right)^{\eta  - 1}}\int_{{Q_j}} {\left| f \right|\sigma d\mu } \le C_{C Z} \lambda .
\end{align}
Now, suppose that $\left\{ {Q_j^k} \right\}$ is the Calderón-Zygmund cubes at height $a^k$ for each $k \in \mathbb{Z}$ and $a > C_{CZ}$. These sets, $E_j^k:=Q_j^k \setminus X_{\eta,a^{k+1}}^{\mathcal{D}}$, are mutually disjoint for all indices $j$ and $k$, such that 
\begin{align}\label{sigema}
\left( {1 - {{\left( {\frac{{{C_{cz}}}}{a}} \right)}^{\frac{1}{{1 - \eta }}}}} \right)\sigma (Q_j^k) \le \sigma (E_j^k) \le \sigma (Q_j^k).
\end{align}

If set $\mu(X) < \infty$, then Calderón-Zygmund cubes can be established for every function $f \in L_{loc}^1(\sigma)$ at any height $\lambda >\lambda_0:=\int_X {\left| f \right|\sigma d\mu }$, meanwhile, \eqref{CZ_1} also holds. Under these conditions, the sets $E_j^k$ are pairwise disjoint with \eqref{sigema} holds, for $k > \log_a \lambda_0$.
\end{lemma}
\begin{proof}

The first case is that $\mu (X) = \infty $. We only need to consider  $X_{\eta,\lambda}^D \ne \emptyset.$ Otherwise, we can take $\left\{ {{Q_j}} \right\}$ to be the empty sets.
	
As the property of the dyadic cube in Theorem \ref{cubes}, for every $x \in X_{\eta,\lambda}^D$, there exists a dyadic cube ${{Q_k^x}}$ of each generation $k > 0$, such that $x \in {{Q_k^x}}$ and $M_{\eta,\sigma}^{\mathcal{D}} f(x) > \lambda$. So there exist $k$, such that 
\begin{align}\label{CZ_2}
\sigma(Q_k^x)^{\eta-1}\int_{Q_k^x}|f(y)| d \sigma>\lambda.
\end{align}
Since
$
\lim _{k \rightarrow \infty} \sigma(Q_k^x)^{\eta-1}\int_{Q_k^x}|f(y)| d \sigma = 0,
$
then there are only finite k such that \eqref{CZ_2} holds.
Select $k$ to be the smallest integer such that \eqref{CZ_2} holds, in this case, we denote the cube with generation $k$ by $Q_x$. What's  more, the set $\left\{ {{Q_x}:x \in X_{\eta, \lambda}^{\mathcal{D}}} \right\}$ can be enumerated as $\left\{ {{Q_j}} \right\}$ due to there are countable dyadic cubes. If ${Q_i} \cap {Q_j} \ne \emptyset $, without loss of generality, we define ${Q_i} \subseteq {Q_j}$. Moreover, by the maximality, ${Q_i} = {Q_j}$.
Thus, the set $\left\{ {{Q_x}:x \in X_{\eta, \lambda}^{\mathcal{D}}} \right\}:=\left\{ {{Q_j}} \right\}$ is countably non-overlapping maximal dyadic cubes.
Hence, $X_{\eta, \lambda}^{\mathcal{D}}\subseteq \bigcup_j Q_j $.

On the other hand, if $z \in Q_x$, for some $x \in X_{\eta, \lambda}^{\mathcal{D}}$, then 
$$
\lambda < \sigma(Q_x)^{\eta-1}\int_{Q_x}|f(y)| d \sigma\le M_{\eta,\sigma}^{\mathcal{D}} f(z).
$$
Thus, $X_{\eta, \lambda}^{\mathcal{D}}= \bigcup_j Q_j $.

Next, we will prove \eqref{CZ_1}. The left inequality of \eqref{CZ_1} holds since the choice of $Q_j$. For the second inequality, by the maximality of each $Q_j$, we can deduce that its parent ${\tilde Q}_j$ satisfies
$$
\sigma({\tilde Q}_j)^{\eta-1}\int_{Q_j}|f(y)| d \sigma \leq \lambda
$$
It follows from Lemmas \ref{cubes} and \ref{LMB.} that
$$
\sigma(Q_j)^{\eta-1}\int_{Q_j}|f(y)| d \sigma \leq {\left( {\frac{{\sigma \left( {{{\tilde Q}_j}} \right)}}{{\sigma \left( {{Q_j}} \right)}}} \right)^{1 - \eta }} \lambda 
\leq {\left( {\frac{{\sigma \left( {B\left( {{x_c}\left( {{{\tilde Q}_j}} \right),Cd_0^{k + 1}} \right)} \right)}}{{\sigma \left( {B\left( {{x_c}\left( {{Q_j}} \right),d_0^k} \right)} \right)}}} \right)^{1 - \eta }} \lambda \leq {\left( {Cd_0^{{{\log }_2}{C_\mu }}} \right)^{1 - \eta }} \lambda.
$$
Consequencely, $\eqref{CZ_1}$ holds.

Setting $a > C_{CZ}$, we define the Calderón-Zygmund cubes $\{Q_j^k\}$ at heights $a^k$ for $k \in \mathbb{Z}$. We abbreviate $ X_{\eta,a^k}^{\mathcal{D}}$ to $X_k $. Given $Q_i^{k+1}$ and for any $x \in Q_i^{k+1}$, we have $Q_i^{k+1} \in \{Q_k^x\}$ (defined as above). It follows that there must be an index $j$ for which $Q_i^{k+1} \subseteq Q_j^k$.

Next, we want to show that the $E_j^k$ are pairwise disjoint for all $j,k$. Setting $k_1 \le k_2$, it suffices to prove that $E_{{j_1}}^{{k_1}} \cap E_{{j_2}}^{{k_2}} = \emptyset$ for $E_{{j_1}}^{{k_1}} \ne E_{{j_2}}^{{k_2}}$. If $k_1=k_2$ and $j_1\ne j_2$, then $Q_{{j_1}}^{{k_1}} \cap Q_{{j_2}}^{{k_2}} = \emptyset$ can deduce the desired results. 
If $k_1 < k_2$, then $E_{{j_1}}^{{k_1}} \subseteq {\left( {{X_{{k_1} + 1}}} \right)^c} \subseteq {\left( {{X_{{k_2}}}} \right)^c}$ and $E_{{j_2}}^{{k_2}} \subseteq {X_{{k_2}}}$ can deduce the desired results.

Finally, we will prove that $\sigma (Q_j^k) \approx \sigma (E_j^k)$.
It follows obviously from that
\begin{align*}
\sigma {(Q_j^k \cap {X_{k + 1}})^{1 - \eta }} &= {\left( {\sum\limits_{i:Q_i^{k + 1} \subseteq Q_j^k} {\sigma (Q_i^{k + 1})} } \right)^{1 - \eta }} \le \sum\limits_{i:Q_i^{k + 1} \subseteq Q_j^k} {{{\left( {\sigma (Q_i^{k + 1})} \right)}^{1 - \eta }}} \\
 &\le \frac{1}{{{a^{k + 1}}}}\sum\limits_{i:Q_i^{k + 1} \subseteq Q_j^k} {\int_{Q_i^{k + 1}} {\left| f \right|d\sigma } } \le \frac{1}{{{a^{k + 1}}}}\int_{Q_j^k} {\left| f \right|d\sigma }  \le \frac{{{C_{cz}}}}{a}\sigma {(Q_j^k)^{1 - \eta }}.
\end{align*}
Note that $\sigma (Q_j^k) = \sigma (Q_j^k \cap {X_{k + 1}}) + \sigma (E_j^k)$, then we have
$$\left( {1 - {{\left( {\frac{{{C_{cz}}}}{a}} \right)}^{\frac{1}{{1 - \eta }}}}} \right)\sigma (Q_j^k) \le \sigma (E_j^k) \le \sigma (Q_j^k).$$
\end{proof}

\section{\bf The proof of Theorems \ref{mainthm_0} and \ref{mainthm_1}}\label{proof}
\subsection{Necessity}\label{proof1}
~
In this subsection, we want to prove the Necessity of Theorem \ref{mainthm_0}. But actually, we prove the stronger claim, which is the necessity of Theorem \ref{mainthm_1}. In this proof, somewhere, we will use the sufficiency of Theorem \ref{mainthm_1}, whose proof can be referred to the next subsection \ref{proof2}. Now, we supppose that 
$M_\eta$ is bounded from $L^{p(\cdot)}(X, \omega)$ to $WL^{q(\cdot)}(X, \omega)$, which means that 
\begin{align}\label{WB-M}
\mathop {\sup }\limits_{t > 0} {\left\| {t\omega {\chi _{\left\{ {x \in X:{M_\eta }f(x) > t} \right\}}}} \right\|_{q( \cdot )}} \lesssim {\left\| {\omega f} \right\|_{p( \cdot )}}.
\end{align}
For following, it suffices to prove that $\omega \in A_{p(\cdot), q(\cdot)}$.

Firstly, we claim that for every $B \subseteq X$, 
\begin{align}\label{Nes.ineq.1}
{\left\| {\omega {\chi _B}} \right\|_{q( \cdot )}} < \infty.
\end{align} 
If ${\left\| {\omega {\chi _B}} \right\|_{q( \cdot )}} = \infty$. For any $x \in B$, there exist $E\subseteq B$, such that $x \in E$.
For any $t<\mu {(B)^{\eta  - 1}}\mu (E)$, then ${M_\eta }\chi_E(x) \ge \mu {(B)^{\eta  - 1}}\mu (E){\chi _B}(x)>t$. Moreover, it follows from \eqref{WB-M} that
$$
\infty = t{\left\| {\omega {\chi _B}} \right\|_{q( \cdot )}} \le {\left\| {t\omega {\chi _{\left\{ {x \in X:{M_\eta }\chi_E(x) > t} \right\}}}} \right\|_{q( \cdot )}} \lesssim {\left\| {\omega {\chi _E}} \right\|_{p( \cdot )}} \lesssim \mu {(E)^\eta }{\left\| {\omega {\chi _E}} \right\|_{q( \cdot )}} \le \infty.
$$

By Lemma~\ref{p.omega}, we have 
$$\mu {(E)^{ - 1}}\int_E \omega  {(x)^{q(x)}}{\mkern 1mu} d\mu  = \infty.$$
When $ E \to \left\{ x \right\}$, by the Lebesgue Differentiation Theorem, see \cite[Theorem 1.4]{Al2015}, we can find that $\omega(x)^{q(x)} = \infty$ for almost every $x$. This result clearly contridicts with the definition of a weight and therefore \eqref{Nes.ineq.1} is valid.

Secondly, we will show that $\omega \in A_{p( \cdot ),q( \cdot )}$. 

{\bf{Case 1: ${\left\| {{\omega ^{ - 1}}{\chi _B}} \right\|_{p'( \cdot )}}< \infty$.}} 

In this case, since the homogeneity, we assume that ${\left\| {{\omega ^{ - 1}}{\chi _B}} \right\|_{p'( \cdot )}}=1$. It suffices to prove that 
\begin{align}\label{Nes.ineq.2}
\mathop {\sup }\limits_{B \subseteq X} \mu {(B)^{\eta  - 1}}{\left\| {\omega {\chi _B}} \right\|_{q( \cdot )}} \lesssim 1
\end{align}

We define the following sets
\[ B_0 \equiv \{x \in B\,:\, p'(x) < \infty\}, \qquad B_\infty \equiv \{x \in B\,:\,p'(x) = \infty\}. \]
By the definition of the norm, for any $\lambda \in (\frac{1}{2},1)$,
$$1 \le {\rho _{p'( \cdot )}}\left( {\frac{{{\omega ^{ - 1}}{\chi _B}}}{\lambda }} \right) = \int_{{B_0}} {{{\left( {\frac{{\omega {{(x)}^{ - 1}}}}{\lambda }} \right)}^{p'(x)}}} {\mkern 1mu} d\mu  + {\lambda ^{ - 1}}{\left\| {{\omega ^{ - 1}}{\chi _{{B_\infty }}}} \right\|_\infty }.$$
At least one of the two terms on the right-hand side is not less than $\frac{1}{2}$. Furthermore, one of the following two situations must be true: either ${\left\| {{\omega ^{ - 1}}{\chi _{{B_\infty }}}} \right\|_\infty } \geq \frac{1}{2}$, or given $\lambda_0 \in (\frac{1}{2},1)$, then $\int_{B_0}\left(\frac{\omega(x)^{-1}}{\lambda}\right)^{p'(x)}\,d\mu \geq \frac{1}{2}$ for any $\lambda \in [\lambda_0,1)$.

Suppose that the first situation holds.
Set {\color{orange} 
$s > {\left\| {{\omega ^{ - 1}}{\chi _{{B_\infty }}}} \right\|_\infty } = \operatorname{essinf}\limits_{x \in {B_\infty }} \omega (x)$
}, there exists a subset $E \subseteq B_\infty$ with $\mu(E)>0$, such that $\mu {(E)^{ - 1}}\omega (E) \le s$. Note that $\pp$ is equal to 1 on $B_\infty$,
then ${\left\| {\omega {\chi _E}} \right\|_{p( \cdot )}} = \omega (E)$.
Then, for all $t < \mu {(B)^{\eta  - 1}}\mu (E)$, we have
${M_\eta }{\chi _E}(x) \ge \mu {(B)^{\eta  - 1}}\mu (E){\chi _B}(x)>t{\chi _B}(x)$.
Thus, it follows from \eqref{WB-M} that  
$$t{\left\| {\omega {\chi _B}} \right\|_{q( \cdot )}} \le {\left\| {t\omega {\chi _{\left\{ {x \in X:{M_\eta }{\chi _E}(x) > t} \right\}}}} \right\|_{q( \cdot )}} \lesssim {\left\| {\omega {\chi _E}} \right\|_{p( \cdot )}} = \omega (E).$$
Letting $t\to \mu {(B)^{\eta  - 1}}\mu (E)$, we get that
$\mu {(B)^{\eta  - 1}}\mu (E)\parallel \omega {\chi _B}{\parallel _{q( \cdot )}} \lesssim \omega (E).$
Then,
\[\mu {(B)^{\eta  - 1}}\parallel \omega {\chi _B}{\parallel _{q( \cdot )}} \lesssim \mu {(E)^{- 1}}\omega (E) \le s\]
Letting $s \to \parallel {\omega ^{ - 1}}{\chi _{{B_\infty }}}\parallel _\infty ^{ - 1}$, we have
\[\mu {(B)^{\eta  - 1}}{\left\| {\omega {\chi _B}} \right\|_{q( \cdot )}} \lesssim \left\| {{\omega ^{ - 1}}{\chi _{{B_\infty }}}} \right\|_\infty ^{ - 1} \le 2,\]
then \eqref{Nes.ineq.2} is valid.

When the second situation holds, we define $B_R = \{x \in B_0\,:\,p'(x) < R\}$, for any $R>1$. By Lemma~\ref{Levi}, there exists $R$ that close to $\infty$ sufficiently, such that $\int_{B_R} \left(\frac{\omega(x)^{-1}}{\lambda_0}\right)^{p'(x)}\,d\mu > \frac{1}{3}.$
It follows from ${\left\| {{\omega ^{ - 1}}{\chi _B}} \right\|_{p'( \cdot )}} = 1$ and Lemma~\ref{p.omega} that
\begin{align*}
	\int_{B_R}\left(\frac{\omega(x)^{-1}}{\lambda_0}\right)^{p'(x)}\,d\mu
	\le \int_{B_R}\left(\frac{2}{\lambda_0}\right)^{p'(x)}\left(\frac{\omega(x)^{-1}}{2}\right)^{p'(x)}\,d\mu  \le \left(\frac{2}{\lambda_0}\right)^R < \infty.
\end{align*}
We need to use the following auxiliary function
\[ G(\lambda) = \int_{B_R} \left(\frac{\omega(x)^{-1}}{\lambda}\right)^{p'(x)}\,d\mu, \]
where $\frac{1}{3} < G(\lambda_0) < \infty$. The Lebesgue dominated convergence theorem can deduce that $G$ is continuous on $[\lambda_0,1]$.

For any $\lambda \in [\lambda_0,1)$, if $G(1) \geq \frac{1}{3}$, by Lemma~\ref{p.omega}, 
\[ \frac{1}{3\lambda} \leq \frac{1}{\lambda}\int_{B_R}\omega(x)^{-p'(x)}\,d\mu \leq G(\lambda) \leq \lambda^{-R} < \infty. \]
Let $\lambda$ sufficiently close to 1, then $\lambda^{-R} \leq 2$ and
 \begin{align} 
	\frac{1}{3} \leq \int_{B_R}\left(\frac{\omega(x)^{-1}}{\lambda}\right)^{p'(x)}\,d\mu \leq 2. \label{finally} 
\end{align}
If $G(1) < \frac{1}{3}$, by continuity of $G$, there exists $\lambda \in (\lambda_0,1)$ such that $G(\lambda)=\frac{1}{3}$. Then \eqref{finally} holds for this $\lambda$ as well.

Fixed $\lambda$ and let 
$$f(x) = \omega {(x)^{ - p'(x)}}{\lambda ^{1 - p'(x)}}{\chi _{{B_R}}}.$$
Then
\[ \rho_{\pp}(\omega f) = \int_{B_R}\left(\frac{\omega(x)^{-1}}{\lambda}\right)^{p'(x)}\,d\mu  \leq 2. \]
By Lemma~\ref{p.omega}, ${\left\| {\omega f} \right\|_{p( \cdot )}} \le {2^{\frac{1}{{{{(p')}_ - }}}}}$. For any $x \in B$,
\[ {M_\eta }f(x) \ge \mu {(B)^{\eta -1  }}\int_B {fd\mu}  =  
\lambda \mu {(B)^{\eta -1  }}\int_{{B_R}} {{{(\frac{{\omega {{(x)}^{ - 1}}}}{\lambda })}^{p'(x)}}d\mu} 
\geq \frac{\lambda }{3}\mu {(B)^{\eta  - 1}}. \]
For any $t < \frac{\lambda }{3}\mu {(B)^{\eta  - 1}}$, it follows from \eqref{WB-M}  that
\[t{\left\| {\omega {\chi _B}} \right\|_{q( \cdot )}} \le {\left\| {t\omega {\chi _{\left\{ {x \in X:{M_\eta }f(x) > t} \right\}}}} \right\|_{q( \cdot )}} \lesssim {\left\| {\omega f} \right\|_{p( \cdot )}} \le {2^{\frac{1}{{{{(p')}_ - }}}}}.\]
Letting $t \to \frac{\lambda }{3}\mu {(B)^{\eta  - 1}}$, \eqref{Nes.ineq.2} is valid.

{\bf{Case 2: ${\left\| {{\omega ^{ - 1}}{\chi _B}} \right\|_{p'( \cdot )}}= \infty$.}}

In this case, we will use the perturbation method to prove.

Given $\epsilon > 0$, denote the weight $\omega_\epsilon(x) = \omega(x)+\epsilon$. Then $\omega_\epsilon^{-1} \leq \epsilon^{-1} < \infty$ and so ${\left\| {{\omega_{\epsilon} ^{ - 1}}{\chi _B}} \right\|_{p'( \cdot )}} < \infty$. It follows immediately from the sufficiency of Theorem \ref{mainthm_0}, Lemma \ref{1}, and \eqref{WB-M} that
\begin{align*} 
t{\left\| {{\omega _\epsilon}{\chi _{\left\{ {x \in X:{M_\eta }f(x) > t} \right\}}}} \right\|_{q( \cdot )}} &\le t{\left\| {{\omega}{\chi _{\left\{ {x \in X:{M_\eta }f(x) > t} \right\}}}} \right\|_{q( \cdot )}} + {\epsilon}t{\left\| {{\chi _{\left\{ {x \in X:{M_\eta }f(x) > t} \right\}}}} \right\|_{q( \cdot )}}\\
&\lesssim {\left\| {\omega f} \right\|_{p( \cdot )}} + \epsilon{\left\| { f} \right\|_{p( \cdot )}} \le 2 {\left\| {\omega_\epsilon f} \right\|_{p( \cdot )}}
\end{align*}
This shows that $\omega_\epsilon$ satisfies \eqref{WB-M}. When ${\left\| {{\omega ^{ - 1}}{\chi _B}} \right\|_{p'( \cdot )}} < \infty$, it follows from \eqref{Nes.ineq.2} that $$\mathop {\sup }\limits_{B \subseteq X} \mu {(B)^{\eta  - 1}}{\left\| {\omega_\epsilon {\chi _B}} \right\|_{q( \cdot )}}{\left\| {{\omega_\epsilon ^{ - 1}}{\chi _B}} \right\|_{p'( \cdot )}} \le K,$$
where $K$ is actually independent of $\epsilon$. Thus,
\[\mu {(B)^{\eta  - 1}}{\left\| {\omega {\chi _B}} \right\|_{q( \cdot )}}{\left\| {{\omega_\epsilon ^{ - 1}}{\chi _B}} \right\|_{p'( \cdot )}} \le \mu {(B)^{\eta  - 1}}{\left\| {\omega_\epsilon {\chi _B}} \right\|_{q( \cdot )}}{\left\| {{\omega_\epsilon ^{ - 1}}{\chi _B}} \right\|_{p'( \cdot )}} \le K. \]
Letting $\epsilon \to 0$, by Lemma \ref{Levi} and $\omega_\epsilon^{-1}$ increases to $\omega^{-1}$, we have that ${\left[ \omega  \right]_{{A_{p( \cdot ),q( \cdot )}}}} \le K$.

This finishes the necessity of Theorems \ref{mainthm_0} and \ref{mainthm_1}.
\subsection{Sufficiency}\label{proof2}
~

The purpose of this section is to prove the sufficiency of Theorem \ref{mainthm_0}, which implicits the sufficiency of Theorem \ref{mainthm_1}. 
We will discuss the case for $\mu(X) < \infty$ at the end of this subsection. The initial focus will be on cases where $\mu(X) = \infty$.

{\bf{Case 1: $\mu(X) = \infty$.}} 
We first simplify some details with three steps.

{\bf{Step 1.}} 
Lemma \ref{Suff_1} implies that to establish the boundedness of $M_\eta$, it is sufficient to demonstrate the boundedness of $M_\eta^{\mathcal{D}}$. By the homogeneity, it suffices to consider that $f$ is a nonnegative function with $\|\omega f \|_{p(\cdot)}=1$. 

{\bf{Step 2.}} 
We introduce the weights $W(\cdot) = \omega(\cdot)^{q(\cdot)}$ and $\sigma(\cdot) = \omega(\cdot)^{-p'(\cdot)}$. According to Lemma \ref{Apq_Ainfty} and Lemma \ref{Ainfty},  $W(\cdot)$ and $\sigma(\cdot)$ are both in $A_{\infty}$ and satisfy the doubling property.

{\bf{Step 3.}} We also need to show that for any nested sequence $\left\{Q_k\in \mathcal{D}\right\}_{k=1}^{\infty}$ with $Q_{k}$ is a child of $Q_{k+1}$,
\begin{equation}\label{Suff.lim}
\mathop {\lim }\limits_{k \to \infty } \mu {({Q_k})^{\eta  - 1}}\int_{{Q_k}} f d\mu  = 0,
\end{equation}
which can guarantee us to use Lemma \ref{CZD}.

Indeed, since $W$ is doubling, if we fix a sequence with $k=1$, then 
\begin{equation*}
    W\left(Q_1\right) \leq W\left(B\left(x_c\left(Q_1\right), C_d d_0\right)\right) \leq C_W^{\log _2 C_d} W\left(B\left(x_c\left(Q_1\right), d_0\right)\right).
\end{equation*}
By Lemma \ref{cubes}, for any $k$, with the similar argument, we have
$$
\frac{1}{W\left(Q_k\right)} \lesssim \frac{1}{W\left(B\left(x_c\left(Q_k\right), C_d d_0^k\right)\right)}.
$$
Using lemma \ref{Ainfty} combining above two estimates, we get
$$
\frac{W\left(Q_1\right)}{W\left(Q_k\right)} \lesssim \frac{W\left(B\left(x_c\left(Q_1\right), d_0\right)\right)}{W\left(B\left(x_c\left(Q_k\right), C_d d_0^k\right)\right)} \lesssim \left(\frac{\mu\left(B\left(x_c\left(Q_1), d_0\right.\right)\right.}{\mu\left(B\left(x_c\left(Q_k\right), C_d d_0^k\right)\right)}\right)^\delta.
$$
If we rearrange and apply Lemma \ref{LMB.} (the lower mass bound), then
$$\mu\left(B\left(x_c\left(Q_1\right), C d_0^k\right)\right)^\delta \lesssim \mu\left(B\left(x_c\left(Q_k\right), C_d d_0^k\right)\right)^\delta \lesssim W\left(Q_k\right).$$
By continuity of $\mu$ and the fact that $X =\mathop {\lim }\limits_{k \to \infty } B\left( {{x_c}\left( {{Q_1}} \right),Cd_0^k} \right)$, we have $\mathop {\lim }\limits_{k \to \infty } W\left( {{Q_k}} \right)=\infty$. 

By the condition of $A_{p(\cdot),q(\cdot)}$, Lemma \ref{Holder}, and Lemma \ref{p.omega},
$$
\mu {\left( {{Q_k}} \right)^{\eta  - 1}}\int_{{Q_k}} {fd\mu }  \lesssim {\left[ \omega  \right]_{{A_{p( \cdot ),q( \cdot )}}}}{\left\| {\omega f} \right\|_{p( \cdot )}}\left\| {\omega {\chi _{{Q_k}}}} \right\|_{q( \cdot )}^{ - 1} \lesssim \left\| {\omega {\chi _{{Q_k}}}} \right\|_{q( \cdot )}^{ - 1}.
$$
Since Lemma \ref{p.omega} implies $\mathop {\lim }\limits_{k \to \infty } W\left( {{Q_k}} \right) = \mathop {\lim }\limits_{k \to \infty } \left\| {\omega {\chi _{{Q_k}}}} \right\|_{q( \cdot )}^{} = \infty$, $\eqref{Suff.lim}$ is valid.

Next, we decompose $f=f_1+f_2$, where $f_1=f \chi_{\left\{f \sigma^{-1}>1\right\}}$ and $f_2=$ $f \chi_{\left\{f \sigma^{-1} \leq 1\right\}}$. Lemma \ref{p.omega} can deduce that
\begin{equation}\label{Suff.ineq_0}
    \int_X\left|f_i(x)\right|^{p(x)} \omega(x)^{p(x)} d \mu \leq\left\|f_i \omega\right\|_{p(\cdot)} \leq \left\|f \omega\right\|_{p(\cdot)} = 1, \quad i=1,2.
\end{equation}

By Lemma \ref{p.omega} again and the sublinearity of $M_\eta^{\mathcal{D}}$, it suffices to show that 
\begin{equation}\label{Suff.ineq_1}
    \int_X\left(M_\eta^{\mathcal{D}} f_i(x)\right)^{q(x)} \omega(x)^{q(x)} d \mu \lesssim 1, \quad i=1,2,
\end{equation}
where the implicit constant is independent on $f$.

{\bf{Estimate for $f_1$:}}
Let $k\in \mathbb{Z}$ and $a>C_{C Z}>1$, and we define
$$
X_k=\left\{x \in X: M_\eta^{\mathcal{D}} f_1(x)>a^{k}\right\} .
$$
Since $f \in L_{ {loc }}^1$ and $\mathop {\lim }\limits_{k \to \infty } \mu {({Q_k})^{\eta  - 1}}\int_{{Q_k}} f d\mu  = 0$, then it follows that $M_\eta^{\mathcal{D}} f_1$ is finite almost everywhere (This result is similar to \cite[Remark 3.11]{red} and \cite[Remark 5.35]{blue}). 

It is apparent to see that
$$
\left\{x \in X: M_\eta^{\mathcal{D}} f_1(x)>0\right\}=\bigcup_{k\in \Z} {X_k}\backslash {X_{k + 1}}.
$$
Let $\{Q_j^k\}$ be the CZ cubes of $f_1$ at height $a^k$. Then by Lemma \ref{CZD}, for every $k$,
\begin{equation}\label{Suff.ineq_2}
{X_k} = \bigcup \limits_j Q_j^k.
\end{equation}
Set $E_j^k=Q_j^k \backslash X_{k+1}$, we find that
$$
X_k \backslash X_{k+1}=\bigcup_j E_j^k.
$$
It is obviously to get that
\begin{align}
&\int_X M_\eta^{\mathcal{D}} f_1(x)^{q(x)} \omega(x)^{q(x)} d \mu \notag\\
 =&\sum_k \int_{X_k \backslash X_{k+1}} M_\eta^{\mathcal{D}} f_1(x)^{q(x)} \omega(x)^{q(x)} d \mu \notag  \\
 \approx& \sum_k \int_{X_k \backslash X_{k+1}} a^{k q(x)} \omega(x)^{q(x)} d \mu \notag \\
 \approx& \sum_{k, j} \int_{E_j^k}\left(\int_{Q_j^k} f_1(y) \sigma(y)^{-1} \sigma(y) d \mu\right)^{q(x)} \mu\left(Q_j^k\right)^{(\eta-1)q(x)} {\omega}(x)^{q(x)} d \mu .\label{Suff.ineq_3}
\end{align}
Through the definition of $f_1,\sigma$ and (\ref{Suff.ineq_0}),
\begin{align*}
&\int_{Q_j^k} f_1(y) \sigma(y)^{-1} \sigma(y) d \mu&\leq \int_{Q_j^k}\left(f_1(y) \sigma(y)^{-1}\right)^{p(y)} \sigma(y) d \mu 
\leq \int_{Q_j^k} (f_1(y)\omega(y))^{p(y)}  d \mu \leq 1.
\end{align*}
Then,
\begin{align*}
& \sum_{k, j} \int_{E_j^k}\left(\int_{Q_j^k} f_1(y) \sigma(y)^{-1} \sigma(y) d \mu\right)^{q(x)} \mu\left(Q_j^k\right)^{(\eta-1)q(x)} {\omega}(x)^{q(x)} d \mu \\
 \leq& \sum_{k, j}\left(\int_{Q_j^k}\left(f_1(y) \sigma(y)^{-1}\right)^{p(y) / p_{-}\left(Q_j^k\right)} \sigma(y) d \mu\right)^{q_{-}\left(Q_j^k\right)} \int_{E_j^k} \mu\left(Q_j^k\right)^{(\eta-1)q(x)} {\omega}(x)^{q(x)} d \mu.
\end{align*}
Next, it follows from Hölder's inequality that the above
\begin{equation}\label{Suff.ineq_4}
\lesssim \sum_{k, j}\left(\avgint_{Q_j^k}\left(f_1(y) \sigma(y)^{-1}\right)^{p(y) / p_{-}} \sigma(y) d \mu\right)^{\frac{{{q_ - }({Q_j^k})}}{{{p_ - }({Q_j^k})}}{p_ - }} \int_{E_j^k} \sigma\left(Q_j^k\right)^{q_{-}\left(Q_j^k\right)} \mu\left(Q_j^k\right)^{(\eta-1)q(x)} {\omega}(x)^{q(x)} d \mu.
\end{equation}
We claim that
\begin{equation}\label{Suff.ineq_55}
\int_{E_j^k} \sigma\left(Q_j^k\right)^{q_{-}\left(Q_j^k\right)} \mu\left(Q_j^k\right)^{(\eta-1)q(x)} {\omega}(x)^{q(x)} d \mu \lesssim \sigma\left(Q_j^k\right)^{\frac{{{q_ - }({Q_j^k})}}{{{p_ - }({Q_j^k})}}}.
\end{equation}

Since $\mu\left(Q_j^k\right) \approx \mu\left(E_j^k\right)$ and $\sigma \in A_{\infty}$, by Lemma \ref{Apq_Ainfty} applied to $\omega^{-1}\in A_{q'(\cdot),p'(\cdot)}$, Lemma \ref{CZD}, and Lemma \ref{Ainfty}, we obtain $\sigma\left(Q_j^k\right) \approx \sigma\left(E_j^k\right)$. Thus, \eqref{Suff.ineq_55} can deduce that \eqref{Suff.ineq_4} is bounded by
\begin{align*} 
&\sum_{k, j}\left(\avgint_{Q_j^k}\left(f_1(y) \sigma(y)^{-1}\right)^{p(y) / p_{-}}  \sigma(y) d\mu\right)^{\frac{{{q_ - }(Q_j^k)}}{{{p_ - }(Q_j^k)}}{p_ - }} \sigma\left(E_j^k\right)^{\frac{{{q_ - }(Q_j^k)}}{{{p_ - }(Q_j^k)}}{p_ - }}\\
\lesssim& \sum_{k,j}\left(\int_{E^k_j} M^\mathcal{D}_\sigma((f_1\sigma^{-1})^{\pp/p_-})(x)^{p_-}\sigma(x)\,d\mu\right)^{\frac{{{q_ - }(Q_j^k)}}{{{p_ - }(Q_j^k)}}} \\
\lesssim&\sum\limits_{\theta  = 1,\frac{{{q_ + }}}{{{p_ - }}}} {\sum\limits_{k,j} {{{\left( {\int_{E_j^k} {M_\sigma ^{\mathcal D}} ({{({f_1}{\sigma ^{ - 1}})}^{p( \cdot )/{p_ - }}}){{(x)}^{{p_ - }}}\sigma (x){\mkern 1mu} d\mu } \right)}^\theta }} }\\
\le&\sum\limits_{\theta  = 1,\frac{{{q_ + }}}{{{p_ - }}}} {{{\left( {\sum\limits_{k,j} {\int_{E_j^k} {M_\sigma ^{\mathcal D}} ({{({f_1}{\sigma ^{ - 1}})}^{p( \cdot )/{p_ - }}}){{(x)}^{{p_ - }}}\sigma (x){\mkern 1mu} d\mu } } \right)}^\theta }}\\
\le&\sum\limits_{\theta  = 1,\frac{{{q_ + }}}{{{p_ - }}}} {{{\left( {\int_X {M_\sigma ^{\mathcal D}} ({{({f_1}{\sigma ^{ - 1}})}^{p( \cdot )/{p_ - }}}){{(x)}^{{p_ - }}}\sigma (x){\mkern 1mu} d\mu } \right)}^\theta }}.
\end{align*}
By Lemma \ref{M.eta.bound} and \eqref{Suff.ineq_1}, the above is bounded by
$\sum\limits_{\theta  = 1,\frac{{{q_ + }}}{{{p_ - }}}} {{{\left( {\int_X {{{(\omega (x){f_1}(x))}^{p(x)}}} d\mu } \right)}^\theta }} \lesssim 1.$

Next, we will verify \eqref{Suff.ineq_55} to finish the estimate for $f_1$.
In fact, the left-hand side of \eqref{Suff.ineq_55} can be rewrite by
\begin{equation}\label{Suff.ineq_6}
	\left(\frac{\sigma\left(Q_j^k\right)}{\left\|\omega^{-1} \chi_{Q_j^k}\right\|_{p^{\prime}(\cdot)}}\right)^{q_{-}\left(Q_j^k\right)} \int_{Q_j^k}\left\|\omega^{-1} \chi_{Q_j^k}\right\|_{p^{\prime}(\cdot)}^{q_{-}\left(Q_j^k\right)-q(x)}\left\|\omega^{-1} \chi_{Q_j^k}\right\|_{p^{\prime}(\cdot)}^{q(x)} \mu\left(Q_j^k\right)^{(\eta-1)q(x)} {\omega}(x)^{q(x)} d \mu.
\end{equation}
To prove \eqref{Suff.ineq_55}, it suffices to prove that
\begin{align}
&\int_{Q_j^k}\left\|\omega^{-1} \chi_{Q_j^k}\right\|_{p^{\prime}(\cdot)}^{q(x)} \mu\left(Q_j^k\right)^{(\eta-1)q(x)} {\omega}(x)^{q(x)} d \mu \lesssim 1,\label{6.10}\\
&\left\|\omega^{-1} \chi_{Q_j^k}\right\|_{p^{\prime}(\cdot)}^{q_{-}\left(Q_j^k\right)-q(x)} \lesssim 1 \label{Suff.ineq_7},\\
&\left(\frac{\sigma\left(Q_j^k\right)}{\left\|\omega^{-1} \chi_{Q_j^k}\right\|_{p^{\prime}(\cdot)}}\right)^{q_{-}\left(Q_j^k\right)} \lesssim \sigma\left(Q_j^k\right)^{\frac{{{q_ - }(Q_j^k)}}{{{p_ - }(Q_j^k)}}} \label{Suff.ineq_8}.
\end{align}

Firstly, \eqref{6.10} follows instantly from the condition of $A_{p(\cdot),q(\cdot)}$ and Lemma \ref{p.omega}.
Secondly, we will prove \eqref{Suff.ineq_7} as follows. 

Assume that $\left\|\omega^{-1} \chi_{Q_j^k}\right\|_{p^{\prime}(\cdot)} <1$, otherwise, there is nothing to prove.
Then,
\begin{align}\label{equiv.1}
p(x) - {p_ - }(Q_j^k) \approx \frac{1}{{{p_ - }(Q_j^k)}} - \frac{1}{{p(x)}} = \frac{1}{{{q_ - }(Q_j^k)}} - \frac{1}{{q(x)}} \approx q(x) - {q_ - }(Q_j^k),
\end{align}

which only depend on $p(\cdot)$ and $\eta$. 
Moreover, we have
\begin{align*}
	q(x)-q_{-}\left(Q_j^k\right) & =\frac{q^{\prime}(x)}{q^{\prime}(x)-1}-\frac{\left(q^{\prime}\right)_{+}\left(Q_j^k\right)}{\left(q^{\prime}\right)_{+}\left(Q_j^k\right)-1} \\
	& =\frac{\left(q^{\prime}\right)_{+}\left(Q_j^k\right)-q^{\prime}(x)}{\left[q^{\prime}(x)-1\right]\left[\left(q^{\prime}\right)_{+}\left(Q_j^k\right)-1\right]} \\
	& \lesssim {\left(q^{\prime}\right)_{+}\left(Q_j^k\right)-\left(q^{\prime}\right)_{-}\left(Q_j^k\right)} \\
	& \approx {\left(p^{\prime}\right)_{+}\left(Q_j^k\right)-\left(p^{\prime}\right)_{-}\left(Q_j^k\right)},
\end{align*}
where the last step holds since we used \eqref{equiv.1} and the implicit constants only depend on $p(\cdot)$ and $\eta$.
Thus, \eqref{Suff.ineq_7} follows immediately from Lemma \ref{fracexp} (applied to cubes) and \eqref{cha.2}.

Last, we prove \eqref{Suff.ineq_8} as follows.
If $\left\|\omega^{-1} \chi_{Q_j^k}\right\|_{p^{\prime}(\cdot)}>1$, then by Lemma \ref{p.omega},
$$
\left(\frac{\sigma\left(Q_j^k\right)}{\left\|\omega^{-1} \chi_{Q_j^k}\right\|_{p^{\prime}(\cdot)}}\right)^{q_{-}\left(Q_j^k\right)} \leq\left(\sigma\left(Q_j^k\right)^{1-1 /\left(p^{\prime}\right)_+\left(Q_j^k\right)}\right)^{q_{-}\left(Q_j^k\right)}=\sigma\left(Q_j^k\right)^{{\frac{{{q_ - }(Q_j^k)}}{{{p_ - }(Q_j^k)}}}} .
$$
If $\left\|\omega^{-1} \chi_{Q_j^k}\right\|_{p^{\prime}(\cdot)} \leq 1$, then applying Lemma \ref{p.omega}  and Lemma \ref{fracexp},
\begin{align*}
	\left(\frac{\sigma\left(Q_j^k\right)}{\left\|\omega^{-1} \chi_{Q_j^k}\right\|_{p^{\prime}(\cdot)}}\right)^{q_{-}\left(Q_j^k\right)} & \leq\left(\left\|\omega^{-1} \chi_{Q_j^k}\right\|_{p^{\prime}(\cdot)}^{\left(p^{\prime}\right)_-\left(Q_j^k\right)-1}\right)^{q_{-}\left(Q_j^k\right)} \\
	& \leq\left(\left\|\omega^{-1} \chi_{Q_j^k}\right\|_{p^{\prime}(\cdot)}^{\left(p^{\prime}\right)_-\left(Q_j^k\right)-1+\left(p^{\prime}\right)_+\left(Q_j^k\right)-\left(p^{\prime}\right)_+\left(Q_j^k\right)}\right)^{q_{-}\left(Q_j^k\right)} \\
	& \lesssim \left(\left\|\omega^{-1} \chi_{Q_j^k}\right\|_{p^{\prime}(\cdot)}^{\left(p^{\prime}\right)_{+}\left(Q_j^k\right)-1}\right)^{q_{-}\left(Q_j^k\right)} \\
	& \lesssim \left(\sigma\left(Q_j^k\right)^{\frac{\left(p^{\prime}\right)_{+}\left(Q_j^k\right)-1}{\left(p^{\prime}\right)_{+}\left(Q_j^k\right)}}\right)^{q_{-}\left(Q_j^k\right)} \\
	& \lesssim \sigma\left(Q_j^k\right)^{{\frac{{{q_ - }(Q_j^k)}}{{{p_ - }(Q_j^k)}}}} .
\end{align*}
Eventually, \eqref{Suff.ineq_55} is valid and then we finish the proof of \eqref{Suff.ineq_0} for $f_1$.

{\bf{Estimate for $f_2$:}}
Initially, we notice that $1, \sigma$, and $W$ are in $A_{\infty}$. Considering $\{Q_j^k\}$ as the Calderón-Zygmund dyadic cubes for $f_2$ relative to $\mu$, and selecting a nested tower of cubes $\{Q_{k, 0}\}$, it is observed that the measures $\mu(Q_{k, 0}), \sigma(Q_{k, 0})$, and $W(Q_{k, 0})$ all tend towards infinity. We will often use the doubling property for $A_\infty$ in following.

Finding a cube $Q_{k_0, 0} =: Q_0 \in \mathcal{D}_{k_0}$ s.t. $\mu\left(Q_0\right), W\left(Q_0\right) $ and $\sigma\left(Q_0\right)\geq 1$ and fixing a $LH_{\infty}$ base point $x_0=x_c\left(Q_0\right)$, by Lemma \ref{LHxy}. Define $N_0=2 A_0 C_d$ and the sets
\begin{align*}
\mathscr{F} & =\left\{(k, j)\in \mathbb{Z}\times \mathbb{Z}: Q_j^k \subseteq Q_0\right\}; \\
\mathscr{G} & =\left\{(k, j)\in \mathbb{Z}\times \mathbb{Z}: Q_j^k \nsubseteq Q_0 \text { and } d\left(x_0, x_c\left(Q_j^k\right)\right)<N_0 d_0^k\right\}; \\
\mathscr{H} & =\left\{(k, j)\in \mathbb{Z}\times \mathbb{Z}: Q_j^k \nsubseteq Q_0 \text { and } d\left(x_0, x_c\left(Q_j^k\right)\right) \geq N_0 d_0^k\right\}.
\end{align*}
By the same argument of getting \eqref{Suff.ineq_3}, and replacing $f_1$ with $f_2$, we have
\begin{align}
\int_X M_{\eta}^{\mathcal{D}} f_2(x)^{q(x)} \omega(x)^{q(x)} d \mu & \lesssim  \sum_{k, j} \int_{E_j^k}\left(\avgint_{Q_j^k} f_2(y) \sigma(y) \sigma(y)^{-1} d \mu\right)^{q(x)} {(\mu\left(Q_j^k\right)^\eta \omega(x))}^{q(x)} d \mu \notag.
\end{align}
We decompose $\sum_{k,j}$ into $\sum_{(k, j) \in \mathscr{F}}=I_1$, $\sum_{(k, j) \in \mathscr{G}}=I_2$ and $\sum_{(k, j) \in \mathscr{H}}=I_3$. 

{\bf{Estimate for $I_1$:}}
Noting that the fact that $f_2 \sigma^{-1} \leq 1$ allows us to remove $f_2$ from consideration. Subsequently, by applying \eqref{Suff.ineq_55}, we obtain
\begin{align*}
I_1 & \leq \sum_{(k, j) \in \mathscr{F}} \int_{E_j^k}\left(\avgint_{Q_j^k} \sigma(y) d \mu\right)^{q(x)} {(\mu\left(Q_j^k\right)^\eta \omega(x))}^{q(x)} d \mu \\
& \leq \sum_{(k, j) \in \mathscr{F}} \int_{E_j^k} \sigma\left(Q_j^k\right)^{q(x)-q_{-}\left(Q_j^k\right)} \sigma\left(Q_j^k\right)^{q_{-}\left(Q_j^k\right)} \mu {\left( {Q_j^k} \right)^{(\eta  - 1)q(x)}}\omega {(x)^{q(x)}} d \mu \\
& \leq \sum_{(k, j) \in \mathscr{F}}\left(1+\sigma\left(Q_j^k\right)\right)^{q_{+}\left(Q_j^k\right)-q_{-}\left(Q_j^k\right)} \int_{E_j^k} \sigma\left(Q_j^k\right)^{q_{-}\left(Q_j^k\right)} \mu {\left( {Q_j^k} \right)^{(\eta  - 1)q(x)}}\omega {(x)^{q(x)}} d \mu \\
& \lesssim \left(1+\sigma\left(Q_0\right)\right)^{q_{+}-q_{-}} \sum_{(k, j) \in \mathscr{F}} \sigma\left(Q_j^k\right)^{\frac{{{q_ - }(Q_j^k)}}{{{p_ - }(Q_j^k)}}} \\
& \lesssim \left(1+\sigma\left(Q_0\right)\right)^{q_{+}-q_{-}} \sum\limits_{\theta  = 1,\frac{{{q_ + }}}{{{p_ - }}}} {\sum\limits_{(k,j) \in {\mathscr{F}}} {\left( {\sigma {{\left( {Q_j^k} \right)}^\theta }} \right)} } \\
& \lesssim \left(1+\sigma\left(Q_0\right)\right)^{q_{+}-q_{-}} \sum\limits_{\theta  = 1,\frac{{{q_ + }}}{{{p_ - }}}} {{{\left( {\sum\limits_{(k,j) \in \mathscr{F}} {\sigma \left( {E_j^k} \right)} } \right)}^\theta }} \\
&\le
\left(1+\sigma\left(Q_0\right)\right)^{q_{+}-q_{-}}\sum\limits_{\theta  = 1,\frac{{{q_ + }}}{{{p_ - }}}} {\sigma {{({Q_0})}^\theta }}.
\end{align*}
where the implicit constants are independent on $Q_j^k$ and $f$.

{\bf{Estimate for $I_2$:}}
Set $B_j^k= B(x_c(Q_j^k), A_0(C_d+1) N_0 d_0^k)$. For $(k,j) \in \mathscr{G}$, as $Q^k_j \not\subseteq Q_0$, if $x_c(Q^k_j) \in Q_0$, then by Lemma \ref{cubes},
$Q_0 \subseteq Q^k_j \subseteq B^k_j.$
If $x_c(Q^k_j) \not\in Q_0$, noting that $Q_0 \supseteq B(x_0,d_0^{k_0})$, we have
\[ d_0^{k_0} \leq d(x_0,x_c(Q^k_j)) \leq N_0d_0^k. \]
By Lemma \ref{cubes} again, since $x_0 \in B(x_c(Q^k_j),N_0d_0^k)$ and $Q_0 \subseteq B(x_0,C_dd_0^{k_0})$, then for every $x \in Q_0,$
\[ d(x,x_c(Q^k_j)) \leq A_0(d(x,x_0)+d(x_0,x_c(Q^k_j))) \leq A_0(C_dd_0^{k_0}+N_0d_0^k) \leq A_0(C_d+1)N_0d_0^k. \]
Hence, for any $(k,j) \in \mathscr{G}$, $Q_0 \subseteq B^k_j$. Furthermore, $W(B^k_j), \sigma(B^k_j) \geq 1$. Note also that by doubling property and Lemma \ref{cubes}, $\mu(Q^k_j) \approx \mu(B^k_j)$. 

Lemma \ref{p.omega} can deduce that ${\left\| {{\omega ^{{\rm{ - }}1}}{\chi _{{Q_0}}}} \right\|_{p'( \cdot )}} \ge 1$, since $\sigma \left( {{Q_0}} \right) \ge 1$. By \eqref{Apq.Ainfty_1}, \eqref{cha.2}, and Lemma \ref{chara.Apq}, it follows that
\begin{align*}
\mu\left(Q_j^k\right)^{-1} \approx \mu\left(B_j^k\right)^{-1} & \lesssim  \mu {\left( {{Q_0}} \right)^{ - 1}}{\left( {\frac{{\sigma \left( {{Q_0}} \right)}}{{\sigma \left( {B_j^k} \right)}}} \right)^{\frac{1}{{(1 - \eta )p_\infty ^\prime }}}}
& \approx\left\|\omega^{-1} \chi_{B_j^k}\right\|_{p^{\prime}(\cdot)}^{{\frac{1}{{(\eta-1   )}}}} \lesssim \left\|\omega^{-1} \chi_{Q_j^k}\right\|_{p^{\prime}(\cdot)}^{{\frac{1}{{(\eta-1 )}}}} .
\end{align*}
Together with the above and Lemma \ref{Holder}, we have
$$
\mu\left(Q_j^k\right)^{\eta-1} \int_{Q_j^k} f_2(y) d \mu \lesssim \left\|\omega^{-1} \chi_{Q_j^k}\right\|_{p^{\prime}(\cdot)}^{-1}\left\|f_2 \omega\right\|_{p(\cdot)}\left\|\omega^{-1} \chi_{Q_j^k}\right\|_{p^{\prime}(\cdot)} \lesssim 1.
$$
It follows immediately from Lemma \ref{2Log_2} that
\begin{align}
I_2 & \lesssim \sum_{(k, j) \in \mathscr{G}} \int_{E_j^k}\left(C^{-1} \mu\left(Q_j^k\right)^{\eta-1} \int_{Q_j^k} f_2(y) d \mu\right)^{q(x)} {\omega}(x)^{q(x)} d \mu \notag \\ \label{Suff.ineq_9}
& \leq C_t \sum_{(k, j) \in \mathscr{G}} \int_{E_j^k}\left(\mu\left(Q_j^k\right)^{\eta-1} \int_{Q_j^k} f_2(y) d \mu\right)^{q_{\infty}} {\omega}(x)^{q(x)} d \mu+\sum_{(k, j) \in \mathscr{G}} \int_{E_j^k} \frac{W(x)}{\left(e+d\left(x_0, x\right)\right)^{t q_{-}}} d \mu .
\end{align}
Similar to getting \eqref{Apq.Ainfty_3}, we can choose t sufficiently large to obtain
\begin{equation}\label{Suff.ineq_10}
    \sum_{(k, j) \in \mathscr{G}} \int_{E_j^k} \frac{W(x)}{\left(e+d\left(x_0, x\right)\right)^{t q_{-}}} d \mu \leq \int_X \frac{W(x)}{\left(e+d\left(x_0, x\right)\right)^{t q_{-}}} d \mu \leq 1.
\end{equation}
To finish the estimation of $I_2$, it suffices to estimate the first term of \eqref{Suff.ineq_9}. Therefore, we have
\begin{align*}
& \sum_{(k, j) \in \mathscr{G}} \int_{E_j^k}\left(\mu\left(Q_j^k\right)^{\eta-1} \int_{Q_j^k} f_2(y) d \mu\right)^{q_{\infty}} {\omega}(x)^{q(x)} d \mu \\
= & \sum_{(k, j) \in \mathscr{G}}\left(\frac{1}{\sigma\left(Q_j^k\right)} \int_{Q_j^k} f_2(y) \sigma(y)^{-1} \sigma(y) d \mu\right)^{q_{\infty}}\left(\frac{\sigma\left(Q_j^k\right)}{\mu\left(Q_j^k\right)^{1-\eta}}\right)^{q_{\infty}} W\left(E_j^k\right) .
\end{align*}
Next, we claim that 
\begin{align}\label{Suff.ineq_11}
{\left( {\frac{{\sigma (Q_j^k)}}{{\mu {{(Q_j^k)}^{1 - \eta }}}}} \right)^{{q_\infty }}} \lesssim \frac{{\sigma (Q_j^k)}}{{W(Q_j^k)}}.
\end{align}

Indeed, applying \eqref{Apq.Ainfty_4} to ($\sigma$, ${p'( \cdot )}$) and ($W$, ${q( \cdot )}$) for cubes, and by $A_{p(\cdot), q(\cdot)}$ condition, it follows that
\begin{align*}
\sigma {(Q_j^k)^{{q_\infty } - 1}} \lesssim \left\| {{\omega ^{ - 1}}{\chi _{Q_j^k}}} \right\|_{p'( \cdot )}^{{q_\infty }} \lesssim {\left( {\frac{{\mu {{(Q_j^k)}^{1 - \eta }}}}{{{{\left\| {\omega {\chi _{Q_j^k}}} \right\|}_{q( \cdot )}}}}} \right)^{{q_\infty }}} \lesssim \frac{{\mu {{(Q_j^k)}^{(1 - \eta ){q_\infty }}}}}{{W(Q_j^k)}}.
\end{align*}
Thus, \eqref{Suff.ineq_11} follows obviously from the rearrangement.

In the following, we proceed to estimate the first term of \eqref{Suff.ineq_9}.

\begin{align} 
&\sum_{(k, j) \in \mathscr{G}}\left(\frac{1}{\sigma\left(Q_j^k\right)} \int_{Q_j^k} f_2(y) \sigma(y)^{-1} \sigma(y) d \mu\right)^{q_{\infty}}\left(\frac{\sigma\left(Q_j^k\right)}{\mu\left(Q_j^k\right)^{1-\eta}}\right)^{q_{\infty}} W\left(E_j^k\right) 
\notag\\
 \lesssim&\sum_{(k, j) \in \mathscr{G}}\left(\frac{1}{\sigma\left(Q_j^k\right)} \int_{Q_j^k} f_2(y) \sigma(y)^{-1} \sigma(y) d \mu\right)^{q_{\infty}} \sigma\left(Q_j^k\right) \notag\\ 
\lesssim&  \sum_{(k, j) \in \mathscr{G}} \int_{E_j^k} M_{\sigma}\left(f_2 \sigma^{-1}\right)(x)^{q_{\infty}} \sigma(x) d \mu \notag \\ 
 \le& \int_X M_{ \sigma}\left(f_2 \sigma^{-1}\right)(x)^{q_{\infty}} \sigma(x) d \mu \label{Suff.ineq_19}\\ 
\lesssim &  \int_X\left(f_2(x) \sigma(x)^{-1}\right)^{q_{\infty}} \sigma(x) d \mu \label{Suff.ineq_14}\\
\le &  \int_X\left(f_2(x) \sigma(x)^{-1}\right)^{p_{\infty}} \sigma(x) d \mu \label{Suff.ineq_13}\\
 \le & C_t \left(\int_X\left(f_2(x) \sigma(x)^{-1}\right)^{p(x)} \sigma(x) d \mu+\int_X \frac{\sigma(x)}{\left(e+d\left(x_0, x\right)\right)^{t p_{-}}} d \mu\right) \label{Suff.ineq_15}\\
=& C_t \left(\int_X f_2(x)^{p(x)} \omega(x)^{p(x)} d \mu+\int_X \frac{\sigma(x)}{\left(e+d\left(x_0, x\right)\right)^{t p_{-}}} d \mu\right).\label{Suff.ineq_17}
\end{align}
where \eqref{Suff.ineq_14} comes from Lemma \ref{M.eta.bound}, \eqref{Suff.ineq_13} holds due to the fact that $f_2\sigma^{-1} \le 1$, and \eqref{Suff.ineq_15} is valid due to Lemma \ref{2Log_2}. Then, the second term of \eqref{Suff.ineq_17} is similar to \eqref{Suff.ineq_10} and we just replace $W$ with $\sigma$. Thus, \( I_2 \) is bounded by a constant due to \eqref{Suff.ineq_0}.

{\bf{Estimate for $I_3$:}} Firstly, we claim that
\begin{equation}\label{Suff.ineq_16}
    \sup _{x \in Q_j^k} d\left(x_0, x\right) \approx \inf _{x \in Q_j^k} d\left(x_0, x\right),
\end{equation}
where the implicit constant is independent on $Q_j^k$. In our analysis, the validity of inequality \eqref{Suff.ineq_16} will be established through substitution of \( Q_j^k \) with the ball \( A_j^k = N_0^{-1}B_j^k \), which encompasses \( Q_j^k \). For this purpose, we fix a pair \( (k, j) \) within \( \mathscr{H} \) and choose an arbitrary \( x \) from \( A_j^k \). We get that
\begin{align*} d(x,x_0) &\leq A_0[d(x,x_c(Q^k_j))+d(x_0,x_c(Q^k_j))]\leq A_0[C_dd_0^k+d(x_0,x_c(Q^k_j))] \leq \left(A_0+\frac{1}{2}\right)d(x_0,x_c(Q^k_j)). 
\end{align*}
In the other hand,
\begin{align*} d(x_0,x_c(Q^k_j)) &\leq A_0[d(x_0,x)+d(x,x_c(Q^k_j))] = \frac{1}{2}N_0d_0^k + A_0d(x_0,x) \leq \frac{1}{2}d(x_0,x_c(Q^k_j)) + A_0d(x_0,x).
\end{align*}
Then, we obtain that
$$
d\left(x_0, x_c\left(Q_j^k\right)\right) \leq 2 A_0 d\left(x_0, x\right) .
$$
Consequently, \eqref{Suff.ineq_16} holds. 

To proceed with the estimation of \( I_3 \), it becomes necessary to partition \( \mathscr{H} \) into two distinct subsets,
$$
\mathscr{H}_1=\left\{(k, j)\in \mathscr{H}: \sigma\left(Q_j^k\right) \leq 1\right\}, \quad \mathscr{H}_2=\left\{(k, j)\in \mathscr{H}: \sigma\left(Q_j^k\right)>1\right\} .
$$
Initially, we aggregate over \( \mathscr{H}_1 \). Consider \( x_{+} \) within \( \overline{A_j^k} \), chosen such that \( q_{+}(A_j^k) = q(x_{+}) \), a selection made possible by the continuity of \( q(\cdot) \) in \( LH_0 \). Subsequently, in accordance with the \( LH_{\infty} \) criterion and inequality \eqref{Suff.ineq_16}, it holds for almost every \( x \) in \( Q_j^k \) that,
\begin{align*}
0\leq q_{+}\left(Q_j^k\right)-q(x) & \leq\left|q\left(x_{+}\right)-q_{\infty}\right|+\left|q(x)-q_{\infty}\right| \\
& \leq \frac{C_{\infty}}{\log \left(e+d\left(x_0, x_{+}\right)\right)}+\frac{C_{\infty}}{\log \left(e+d\left(x_0, x\right)\right)} \\
& \approx\frac{1}{\log \left(e+d\left(x_0, x\right)\right)}.
\end{align*}

By Lemma \ref{2Log_2} and \eqref{Suff.ineq_10}, we derive
\begin{align}
 &\sum_{(k, j) \in \mathscr{H}_1} \int_{E_j^k}\left(\mu\left(Q_j^k\right)^{\eta-1} \int_{Q_j^k} f_2(y) d \mu\right)^{q(x)} {\omega}(x)^{q(x)} d \mu \notag\\
 \lesssim& \left(\sum_{(k, j) \in \mathscr{H}_1} \int_{E_j^k}\left(\mu\left(Q_j^k\right)^{\eta-1} \int_{Q_j^k} f_2(y) d \mu\right)^{q_{+}\left(Q_j^k\right)} {\omega}(x)^{q(x)} d \mu\right) + 1.
 \label{Suff.ineq_18}
\end{align}
It follows from Lemma \ref{2Log_3} that
$$
\mu\left(Q_j^k\right)^{q(x)-q_{+}\left(Q_j^k\right)} \lesssim\left(\mu\left(Q_j^k\right)^{q_{+}\left(Q_j^k\right)}+\mu\left(Q_j^k\right)^{q_{-}\left(Q_j^k\right)}\right) \mu\left(Q_j^k\right)^{-q_{+}\left(Q_j^k\right)}\lesssim 1.
$$
The first term of \eqref{Suff.ineq_18} is bounded by
$$
\sum_{(k, j) \in \mathscr{H}_1} \int_{E_j^k}\left(\frac{1}{\sigma\left(Q_j^k\right)} \int_{Q_j^k} f_2(y) \sigma(y)^{-1} \sigma(y) d \mu\right)^{q_{+}\left(Q_j^k\right)} \sigma\left(Q_j^k\right)^{q_+\left(Q_j^k\right)} \mu\left(Q_j^k\right)^{(\eta-1) q(x)} {\omega}(x)^{q(x)} d \mu.
$$

Through Lemma \ref{2Log_2} and the fact that $f_2 \sigma^{-1} \leq 1$, the above
\begin{align*}
& \lesssim \sum_{(k, j) \in \mathscr{H}_1} \int_{E_j^k}\left(\frac{1}{\sigma\left(Q_j^k\right)} \int_{Q_j^k} f_2(y) \sigma(y)^{-1} \sigma(y) d \mu\right)^{q_{\infty}} \sigma\left(Q_j^k\right)^{q_{+}\left(Q_j^k\right)} \mu\left(Q_j^k\right)^{(\eta-1) q(x)} {\omega}(x)^{q(x)} d \mu \\
& + \sum_{(k, j) \in \mathscr{H}_1} \int_{E_j^k} \sigma\left(Q_j^k\right)^{q_{+}\left(Q_j^k\right)} \mu\left(Q_j^k\right)^{(\eta-1) q(x)} \frac{{\omega}(x)^{q(x)}}{\left(e+d\left(x_0, x\right)\right)^{t q_{-}}} d \mu\\ 
&=:J_1+J_2.
\end{align*}
To estimate \( J_2 \), we note that $\sigma (E_j^k) \approx \sigma(Q_j^k) \leq 1$. By \eqref{Suff.ineq_55} and \eqref{Suff.ineq_16}, we deduce that
\begin{align*}
J_2 & \leq \sum_{(k, j) \in \mathscr{H}_1} \sup _{x \in E_j^k}\left(e+d\left(x_0, x\right)\right)^{-t q_{-}} \int_{E_j^k} \sigma\left(Q_j^k\right)^{q_{-}\left(Q_j^k\right)} \mu\left(Q_j^k\right)^{(\eta-1) q(x)} {\omega}(x)^{q(x)} d \mu \\
& \lesssim  \sum_{(k, j) \in \mathscr{H}_1} \sup _{x \in E_j^k}\left(e+d\left(x_0, x\right)\right)^{-t q_{-}} \sigma\left(E_j^k\right) \\
& \lesssim  \sum_{(k, j) \in \mathscr{H}_1} \int_{E_j^k} \frac{\sigma(x)}{\left(e+d\left(x_0, x\right)\right)^{t q_{-}}} d \mu \\
& \le\int_X \frac{\sigma(x)}{\left(e+d\left(x_0, x\right)\right)^{t q_{-}}} d \mu\\
& \lesssim 1.
\end{align*}
where the last inequality is the same as the argument for estimating the second term in \eqref{Suff.ineq_17}.
Similarly, it follows obviously from \eqref{Suff.ineq_4} and \eqref{Suff.ineq_55} that
\begin{align*}
J_1 &\lesssim \sum_{(k, j) \in \mathscr{H}_1}\left(\sigma\left(Q_j^k\right)^{-1} \int_{Q_j^k} f_2(y) \sigma(y)^{-1} \sigma(y) d \mu\right)^{q_{\infty}} \sigma\left(Q_j^k\right)^{\frac{{{q_ - }(Q_j^k)}}{{{p_ - }(Q_j^k)}}}\\ 
& \lesssim \sum_{(k, j) \in \mathscr{H}_1}\left(\sigma\left(Q_j^k\right)^{-1} \int_{Q_j^k} f_2(y) \sigma(y)^{-1} \sigma(y) d \mu\right)^{q_{\infty}} \sigma\left(E_j^k\right) \\
& \lesssim \int_X M_{\sigma}\left(f_2 \sigma^{-1}\right)(x)^{q_{\infty}} \sigma(x) d \mu.
\end{align*}
where the last estimate similar to \eqref{Suff.ineq_19}, which is bounded by a constant.
We finish the estimate for \( \mathscr{H}_1 \). 

Finally, for the case of \( \mathscr{H}_2 \), by Lemma \ref{Holder}, we have
$$
\int_{Q_j^k} f_2(y) d \mu \lesssim \left\|f_2 \omega\right\|_{p(\cdot)}\left\|\omega^{-1} \chi_{Q_j^k}\right\|_{p^{\prime}(\cdot)} \le \left\|\omega^{-1} \chi_{Q_j^k}\right\|_{p^{\prime}(\cdot)} .
$$

Applying Lemma \ref{2Log_2}, 
\begin{align*}
& \sum_{(k, j) \in \mathscr{H}_2} \int_{E_j^k}\left(\avgint_{Q_j^k} f_2(y) d \mu\right)^{q(x)} (\mu {(Q_j^k)^\eta }{\omega}(x))^{q(x)} d \mu \\
 \lesssim& \sum_{(k, j) \in \mathscr{H}_2}\int_{E_j^k} {{{\left( {c\left\| {{\omega ^{ - 1}}{\chi _{Q_j^k}}} \right\|_{{p^\prime }( \cdot )}^{ - 1}\int_{Q_j^k} {{f_2}} (y)d\mu } \right)}^{q(x)}}} \left\| {{\omega ^{ - 1}}{\chi _{Q_j^k}}} \right\|_{{p^\prime }( \cdot )}^{q(x)}\mu {(Q_j^k)^{(\eta  - 1)q(x)}}\omega {(x)^{q(x)}}d\mu \\
 \lesssim& \sum_{(k, j) \in \mathscr{H}_2} \int_{E_j^k} {{{\left( {\left\| {\omega^{-1} {\chi _{Q_j^k}}} \right\|_{{p^\prime }( \cdot )}^{ - 1}\int_{Q_j^k} {{f_2}} (y)d\mu } \right)}^{{q_\infty }}}} \left\| {{\omega ^{ - 1}}{\chi _{Q_j^k}}} \right\|_{{p^\prime }( \cdot )}^{q(x)}\mu {(Q_j^k)^{(\eta  - 1)q(x)}}\omega {(x)^{q(x)}}d\mu \\
&+\sum_{(k,j) \in \mathscr{H}_2} \int_{E_j^k} {\frac{{\left\| {{\omega ^{ - 1}}{\chi _{Q_j^k}}} \right\|_{{p^\prime }( \cdot )}^{q(x)}\mu {{(Q_j^k)}^{(\eta  - 1)q(x)}}\omega {{(x)}^{q(x)}}}}{{{{(e + d({x_0},x))}^{t{q_ - }}}}}{\mkern 1mu} d\mu } \\ 
=:& K_1 + K_2.
\end{align*}
To estimate $K_2$, note that $1 \leq \sigma\left(Q_j^k\right) \approx \sigma\left(E_j^k\right)$. By \eqref{6.10} and \eqref{Suff.ineq_16}, it follows from that
\begin{align}\label{K2}
K_2 & \lesssim  \sum_{(k, j) \in \mathscr{H}_2} \sup _{x \in E_j^k}\left(e+d\left(x_0, x\right)\right)^{-t q_{-}} \int_{E_j^k} {\left\| {{\omega ^{ - 1}}{\chi _{Q_j^k}}} \right\|_{{p^\prime }( \cdot )}^{q(x)}\mu {{(Q_j^k)}^{(\eta  - 1)q(x)}}\omega {{(x)}^{q(x)}}d\mu } \notag\\
& \lesssim \sum_{(k, j) \in \mathscr{H}_2} \sup _{x \in E_j^k}\left(e+d\left(x_0, x\right)\right)^{-t q_{-}} \notag\\
& \lesssim \sum_{(k, j) \in \mathscr{H}_2} \sup _{x \in E_j^k}\left(e+d\left(x_0, x\right)\right)^{-t q_{-}} \sigma\left(E_j^k\right) \notag\\
& \lesssim \int_X \frac{\sigma(x)}{\left(e+d\left(x_0, x\right)\right)^{t q_{-}}} d \mu.
\end{align}
Actually, (\ref{K2}) has been argued in $J_2$ and $I_2$ which is bounded by a constant.

To estimate $K_1$, it follows from \eqref{Apq.Ainfty_4} to get
$$
\left\|w^{-1} \chi_{Q_j^k}\right\|_{p^{\prime}(\cdot)}^{-q_{\infty}} \sigma(Q_j^k) ^{q_{\infty}} \lesssim \sigma(Q_j^k)^{q_{\infty}-q_{\infty} / q_{\infty}^{\prime}}=\sigma(Q_j^k)^{\frac{{{q_\infty }}}{{{p_\infty }}}} .
$$
Since $1 \leq \sigma\left(Q_j^k\right) \approx \sigma\left(E_j^k\right)$, by \eqref{6.10}, we have
\begin{align*}
  K_1
  =&\sum_{(k, j) \in \mathscr{H}_2} \int_{E_j^k} {{{\left( {\sigma {{\left( {Q_j^k} \right)}^{ - 1}}\int_{Q_j^k} {{f_2}} (y)d\mu } \right)}^{{q_\infty }}}\left\| {{\omega ^{ - 1}}{\chi _{Q_j^k}}} \right\|_{{p^\prime }( \cdot )}^{q(x) - {q_\infty }}\sigma {{\left( {Q_j^k} \right)}^{{q_\infty }}}\mu {{(Q_j^k)}^{(\eta  - 1)q(x)}}\omega {{(x)}^{q(x)}}d\mu }\\
   \lesssim& \sum_{(k, j) \in \mathscr{H}_2}\left(\frac{1}{\sigma\left(Q_j^k\right)} \int_{Q_j^k} f_2(y) d \mu\right)^{q_{\infty}} \sigma\left(Q_j^k\right)^{\frac{{{q_\infty }}}{{{p_\infty }}}} \int_{Q_j^k} {\left\| {{\omega ^{ - 1}}{\chi _{Q_j^k}}} \right\|_{{p^\prime }( \cdot )}^{q(x)}} \mu {(Q_j^k)^{(\eta  - 1)q(x)}}\omega {(x)^{q(x)}}d\mu \\
\lesssim& {\left( {\sum\limits_{(k,j) \in {{\mathscr{H}}_2}} {{{\left( {\frac{1}{{\sigma \left( {Q_j^k} \right)}}\int_{Q_j^k} {{f_2}} (y)d\mu } \right)}^{{p_\infty }}}\sigma \left( {E_j^k} \right)} } \right)^{\frac{{{q_\infty }}}{{{p_\infty }}}}} \\
   \lesssim& {\left( {\int_X {{M_{\sigma }}} \left( {{f_2}{\sigma ^{ - 1}}} \right){{(x)}^{{p_\infty }}}\sigma (x)d\mu } \right)^{\frac{{{q_\infty }}}{{{p_\infty }}}}}.
\end{align*}
Next, we use the same method as for estimating \eqref{Suff.ineq_19} to make the above estimate bounded by a constant. Thus the estimates for $I_3$ are completed, which accomplishs the proof of sufficiency for $\mu (X) = \infty$.

{\bf{Case 2: $\mu(X) < \infty$.}} 

Last but not least, we turn to the case for $\mu(X) < \infty$. 
In the finite case, the proof is similar to before and we just need to make some changes for Calderón-Zygmund Decomposition. 
We also consider $f$ is a nonnegative funcion with \(\|\omega f \|_{p(\cdot)} = 1\) and decompose $f=f_1+f_2$ as before.
The construction of Calderón-Zygmund cubes at any height \(\lambda>\lambda_0 := \mu {\left( {Q_j^k} \right)^{\eta  - 1}}\int_{{Q^k}} {f_id\mu }\).  

By Lemma \ref{Holder}, Lemma \ref{finite-bounded}, and the condition of $A_{p(\cdot),q(\cdot)}$,
$$
{\lambda _0} \le 4\mu {(X)^{\eta  - 1}}{\left\| {{f_i}\omega } \right\|_{p( \cdot )}}{\left\| {{\omega ^{ - 1}}} \right\|_{{p^\prime }( \cdot )}} \le 4{\left[ \omega  \right]_{{A_{p( \cdot ),q( \cdot )}}}}\left\| \omega  \right\|_{q( \cdot )}^{ - 1}.
$$
In addition, by Lemma \ref{p.omega} and Lemma \ref{finite-bounded} again, we can conclude that $\lambda_0 \lesssim 1$.

From Lemma \ref{CZD}, set $a=2 C_{C Z}$ and $\left\{ {Q_j^k} \right\}$ is the Calderón-Zygmund cubes of $f_i$ at height $a^k$, for all integers $k \geq k_0=[\log _a \lambda_0]$. Then 
$$
X = X_{\eta ,{a^{{k_0}}}}^{\mathcal D}\bigcup {{{\left( {X_{\eta ,{a^{{k_0}}}}^{\mathcal D}} \right)}^c}}  = \left( {\bigcup\limits_{k = {k_0}}^\infty  {X_{\eta ,{a^k}}^{\mathcal D}\backslash X_{\eta ,{a^{k + 1}}}^{\mathcal D}} } \right)\bigcup {{{\left( {X_{\eta ,{a^{{k_0}}}}^{\mathcal D}} \right)}^c}},
$$
where $X_{\eta,a^{k_0}}^{\mathcal D}:=\left\{x \in X: M_\eta^{\mathcal{D}} f_i(x)>\lambda_0\right\} \subseteq \left\{ {Q_j^k} \right\}$.

It follows instantly from the getting of \eqref{Suff.ineq_3} that
\begin{align*}
&\int_X M_\eta^{\mathcal{D}} f_i(x)^{q(x)} \omega(x)^{q(x)} d \mu \\
 =&\int_{{{{\left( {X_{\eta ,{a^{{k_0}}}}^{\mathcal D}} \right)}^c}}} M_\eta^{\mathcal{D}} f_i(x)^{q(x)} \omega(x)^{q(x)} d \mu+\sum_{k=k_0}^{\infty} \int_{{X_{\eta ,{a^k}}^{\mathcal D}\backslash X_{\eta ,{a^{k + 1}}}^{\mathcal D}}} M_\eta^{\mathcal{D}} f_i(x)^{q(x)} \omega(x)^{q(x)} d \mu \\
 \lesssim & \lambda_0 W(X)+ \sum_{k \geq k_0, j} {\left( {\int_{Q_j^k} {{f_i}} (y)\sigma {{(y)}^{ - 1}}\sigma (y)d\mu } \right)^{q(x)}}\mu {\left( {Q_j^k} \right)^{(\eta  - 1)q(x)}}\omega {(x)^{q(x)}}d\mu.
\end{align*}
The first term is bounded by a constant, which depends only on $X$, $\mathcal D$, $\omega$, $\eta$, and $p(\cdot)$. When $i=1$, the second term is similar to the infinite case. We consider the following for $i=2$.

 After choosing $Q_0 = X$, then $I_2=I_3=0$. Further, since $f_2 \sigma^{-1} \leq 1$, $\sigma(Q_j^k) \approx \sigma(E_j^k)$, and \eqref{Suff.ineq_55}, the second term is bounded by
\begin{align*}
&\quad\sum_{k \geq k_0, j} \int_{E_j^k} \sigma(X)^{q(x)}\left(\frac{\sigma\left(Q_j^k\right)}{\sigma(X)}\right)^{q(x)} \mu {\left( {Q_j^k} \right)^{(\eta  - 1)q(x)}}\omega {(x)^{q(x)}}d \mu \\
& \leq\left(\sigma(X)^{q_{+}}+\sigma(X)^{q_{-}}\right) \sum_{k \geq k_0, j} \int_{E_j^k}\left(\frac{\sigma\left(Q_j^k\right)}{\sigma(X)}\right)^{q_{-}\left(Q_j^k\right)} \mu {\left( {Q_j^k} \right)^{(\eta  - 1)q(x)}}\omega {(x)^{q(x)}}d \mu\\
& \lesssim \left(\sigma(X)^{q_{+}}+\sigma(X)^{q_{-}}\right)\left(\frac{1}{\sigma(X)^{q_+}}+\frac{1}{\sigma(X)^{q_{-}}}\right) \sum_{k \geq k_0, j} \sigma\left(E_j^k\right)^{\frac{{{q_ - }(Q_j^k)}}{{{p_ - }(Q_j^k)}}} \\
& \le \frac{\left(\sigma(X)^{q_{+}}+\sigma(X)^{q_{-}}\right)^2}{\sigma(X)^{q_{+}+q_{-}}} \sum\limits_{\theta  = 1,\frac{{{q_ + }}}{{{p_ - }}}} {\sum\limits_{k \ge {k_0},j} { {\sigma {{\left( {E_j^k} \right)}^\theta }} } } \\
& \le\frac{\left(\sigma(X)^{q_{+}}+\sigma(X)^{q_{-}}\right)^2}{\sigma(X)^{q_{+}+q_{-}}}\sum\limits_{\theta  = 1,\frac{{{q_ + }}}{{{p_ - }}}} {{{\left( {\sum\limits_{k \ge {k_0},j} {\sigma \left( {E_j^k} \right)} } \right)}^\theta }} \\
& \le \frac{{{{\left( {\sigma {{(X)}^{{q_ + }}} + \sigma {{(X)}^{{q_ - }}}} \right)}^2}}}{{\sigma {{(X)}^{{q_ + } + {q_ - }}}}}\sum\limits_{\theta  = 1,\frac{{{q_ + }}}{{{p_ - }}}} {\sigma {{(X)}^\theta }}.
\end{align*}
We accomplish the estimate for $i=2$ and finish the proof of sufficiency for $\mu(X)<\infty$.

\vspace{1cm}
\noindent{\bf Acknowledgements } 
The author would like to thank the editors and reviewers for careful reading and valuable comments, which lead to the improvement of this paper. 

\medskip 

\noindent{\bf Data Availability} Our manuscript has no associated data.



\end{document}